\documentclass[12pt]{amsart}

\usepackage{framed}
\usepackage{amsthm}
\usepackage{verbatim}
\usepackage{amssymb}
\usepackage{enumerate}
\usepackage[active]{srcltx}

\usepackage[utf8]{inputenc}
\numberwithin{equation}{section}
\usepackage{float}
\usepackage{geometry}
\usepackage{tikz}
 \usepackage[usenames,dvipsnames]{pstricks}
 \usepackage{epsfig}
\usepackage{pst-grad} 

\usepackage{filecontents}

\newtheorem{theorem}{Theorem}[section]

\newtheorem{lemma}[theorem]{Lemma}

\newtheorem{obs}[theorem]{Observation}
\newtheorem{clobs}{Observation}[theorem]
\newtheorem{lclaim}{Claim}[theorem]

\newtheorem{ex}[theorem]{Example}
\newtheorem{prob}[theorem]{Problem}
\newtheorem{claim}[theorem]{Claim}

\newtheorem{sclaim}{Subclaim}[theorem]
\newtheorem{conj}[theorem]{Conjecture}

\makeatletter
\newtheorem*{rep@theorem}{\rep@title}
\newcommand{\newreptheorem}[2]{%
\newenvironment{rep#1}[1]{%
 \def\rep@title{#2 \ref{##1}}%
 \begin{rep@theorem}}%
 {\end{rep@theorem}}}
\makeatother
\newreptheorem{theorem}{Theorem}

\DeclareMathOperator{\ih}{IH}

\DeclareMathOperator{\cf}{cf}

\theoremstyle{definition}
\newtheorem{definition}[theorem]{Definition}
\newtheorem{ldefinition}{Definition}[theorem]

\theoremstyle{remark}

\newif\ifdeveloping

\ifdeveloping
\usepackage[notref,notcite]{showkeys}
\fi

\newif\ifcommented
\commentedtrue

\newcommand{\NN}{\mathbb{N}}

\newcommand{\lkd}{\omega\text{-unseparable}}

\newcommand{\smf}{\hspace{0.008 cm}^\smallfrown}

\newcommand{\omg}{{\omega_1}}

\newcommand{\halff}[1]{H_{#1,#1}}

\newcommand{\sat}[1]{{#1}\text{-unseparable}}
\newcommand{\cent}[2]{(#1,#2)\text{-centered}}
\newcommand{\den}[2]{(#1,#2)\text{-dense}}
\newcommand{\utak}[2]{\spadesuit_{#1,#2}}
\newcommand{\utakk}[1]{\spadesuit_{#1}}

\newcommand{\mc}[1]{\mathcal{#1}}
\newcommand{\mbb}[1]{\mathbb{#1}}
\newcommand{\mb}[1]{\mathbf{#1}}

\newcommand{\uhp}{\upharpoonright}
\newcommand{\gr}{G=(V,E)}
\newcommand{\setm}{\setminus}

\newcommand{\subs}{\subset}

\newcommand{\ran}{\operatorname{ran}}
\def\<{\left\langle}
\def\>{\right\rangle}
\def\br#1;#2;{\bigl[ {#1} \bigr]^ {#2} }

\author[D. T. Soukup]{D\'aniel T. Soukup}

\thanks
  {
   } 

\date{\today}
\address{Mathematics \& Statistics, 
University of Calgary
612 Campus Place N.W.
Calgary, AB, Canada
T2N 1N4}

\email{daniel.soukup@ucalgary.ca}
\urladdr{http://www.renyi.hu/$\sim  $dsoukup}

\subjclass[2010]{05C63, 05C70}
\keywords{graph partition, monochromatic path, infinite complete graph, edge colouring}

\title[Monochromatic path decompositions]
   {Decompositions of edge-coloured infinite complete graphs into monochromatic paths II}
\date{\today}
\begin{document}

 \begin{abstract}
We prove that given an edge colouring of an infinite complete graph with finitely many colours, one can partition the vertices of the graph into disjoint monochromatic paths of different colours. This answers a question of R. Rado from 1978.

\end{abstract}
\maketitle

\section{Introduction}

A \emph{path} in a  graph $\gr$ is a 1-1 sequence of vertices $v_0,v_1,v_2\dots$ such that consecutive pairs of vertices form an edge. Now, suppose that one colours the edges of $G$ with finitely many colours, i.e. consider an arbitrary map $c:E\to r$ for some finite $r\geq 1$. Then it makes sense to talk about \textit{monochromatic paths}, i.e. paths $v_0,v_1,v_2\dots$ so that all the edges $\{v_i,v_{i+1}\}$ are coloured identically.

P. Erd\H os proved that if the edges of the complete graph on $\mbb N$ are coloured with two colours, then the vertices can be partitioned into two monochromatic paths of different colours. In this situation, we say that the paths cover $G$. This result appeared in a paper of R. Rado \cite{R} along with a significant generalization: 

\begin{theorem}[{\cite[Theorem 2]{R}}]\label{radostrong} Suppose that $\gr$ is an infinite graph,  $A\subseteq V$ is countable and $$|\{v\in V: \{u,v\}\notin E\}|<|V|$$ for every $u\in A$. If the edges of $G$ are coloured with finitely many colours then $A$ is covered by the vertex set of finitely many disjoint monochromatic paths of different colours.

\end{theorem}

In particular,  Erd\H os' result extends from two colours to an arbitrary finite number of colours. 

Theorem \ref{radostrong} was the starting point of several papers in the past which dealt with similar path decomposition problems either on finite or countably infinite graphs; see \cite{EGyP, Gy, GyS1, Po2, monopathI}. While it is easy to see that every 2-edge coloured finite complete graph is the union of two disjoint monochromatic paths (of different colours) the corresponding question for more colours is significantly harder. Indeed, it was shown only recently by A. Pokrovskiy that every 3-edge coloured finite complete graph is the union of 3 disjoint monochromatic paths (of not necessarily different colours) \cite{Po2}. The case of 4 colours is still completely open; we mention that the currently known best upper bound is given by the following theorem of A. Gy{\'a}rf{\'a}s, M. Ruszink{\'o}, G. N. S{\'a}rk{\"o}zy and E. Szemer{\'e}di:

\begin{theorem}[\cite{bestbound}]\label{bestbound} For every integer $r\geq 2$ there is $n_0(r)\in \mbb N$ such that if $n\geq n_0(r)$ then every $r$-edge coloured copy of $K_n$ can be partitioned into at most $100r\log(r)$ monochromatic cycles.
\end{theorem} 

On generalizations of Theorem \ref{radostrong} to hypergraphs and powers of path on countably infinite vertex sets, we refer the reader to \cite{monopathI}.



At the very end of \cite{R}, Rado introduced a natural extension of paths to the uncountable setting; see Section \ref{notsec} for the definition. Rado asked if the path decomposition result of Erd\H os concerning the complete graph on $\mathbb N$ extends to \emph{uncountable complete graphs of arbitrary size}. 

The goal of this paper is to answer this question affirmatively by proving the following:

\begin{reptheorem}{maindecomp} Suppose that $c$ is a finite-edge colouring of an infinite graph $G=(V,E)$ which satisfies $$|\{v\in V:\{u,v\}\notin E\}|<|V|$$ for all $u\in V$. Then the vertices of $G$ can be partitioned into disjoint monochromatic paths of different colours.
\end{reptheorem}

The smallest uncountable case of this theorem with two colours was essentially proved by M. Elekes, L. Soukup, Z. Szentmiklóssy and the present author \cite{monopathI}; hence the current paper can be considered a continuation of that project (thus the title of our paper). 

Our paper is structured as follows: we start with introducing notations, basic definitions and stating easy observations in Sections \ref{notsec} and \ref{pathconsec}. The proof of Theorem \ref{maindecomp} is preceded by a series of results on finding monochromatic paths in certain classes of graphs. 

First, let us emphasize Lemma \ref{3line} from Section \ref{sec3line}, where we show that any set of vertices $A$ in a graph $G$ which satisfies three rather simple properties can be covered by a path. Next, we prove two important results: Lemma  \ref{ml1} and \ref{ml2}. These lemmas are used in Section \ref{secdich} to show the existence of large sets satisfying all three conditions of Lemma \ref{3line} in certain finite-edge coloured graphs and hence to show the existence of large monochromatic paths; this is done in Theorem \ref{halff}.


In Section \ref{firstdec}, we prove that there is a large family of bipartite graphs $G$ satisfying that for every finite-edge colouring of $G$, we can cover one class of $G$ with disjoint monochromatic paths of different colours. This is done by putting together several lemmas in Theorem \ref{centeredcover}. Finally, after further preparations in Section \ref{firstdec}, the previous results yield the proof of Theorem \ref{maindecomp} in Section \ref{seconddec}.


We believe that the results of this paper are accessible to a wide audience of combinatorists with minimal background in set theory. Furthermore, we hope that the new combinatorial tricks from this paper can be applied to eventually settle the path decomposition problem for finite complete graphs as well.

\subsection*{Acknowledgements} The research presented in the article was mostly carried out at the University of Toronto and the Alfr\'ed R\'enyi Institute of Mathematics. The paper was finalized at the University of Calgary.

 We thank the authors of \cite{monopathI}, A. Dow and W. Weiss for helpful comments and proofreading at various stages of this long project. We are grateful for the referee's careful reading and several useful comments.

\section{Notations} \label{notsec}

A \emph{graph} is an ordered pair $\gr$ so that $E\subseteq [V]^2$; we will use the notation $V(G),E(G)$ for the vertices and edges of a graph $G$. For a graph $\gr$ we write $$N_G(v)=\{w\in V:\{v,w\}\in E\}$$ for $v\in V$ and $$N_G[F]=\bigcap\{N_G(v):v\in F\}$$ for $F\subseteq V$.

 We say that $H$ is a \emph{subgraph} of $G$ iff $V(H)\subseteq V(G)$ and  $E(H)\subseteq E(G)$.

\subsection*{Paths} Clasically, a \emph{path} in a graph $G$ is a 1-1 sequence of vertices $v_0,v_1,\dots$ such that $\{v_i,v_{i+1}\}\in E(G)$. R. Rado introduced a more general definiton which applies for uncountable sequences of vertices as well.

\begin{definition}[R. Rado, \cite{R}] We say that a graph $P$ is a \emph{path} iff there is a well ordering $\prec$ on $V(P)$ such that $$\{w\in N_P(v):w\prec v\} \text{ is }\prec\text{-cofinal below } v$$ for all $v\in V(P)$.
\end{definition}

Given a graph $P=(V,E)$ and well ordering $\prec$ on $V$ we let $$E_\prec=\{(u,v)\in V^2:u\prec v,\{u,v\}\in E\}.$$

\begin{obs} A graph $P$ is a path witnessed by the well ordering $\prec$ iff for all $v\prec w\in V(P)$  there is an injective map $p:n+2\to V(P)$ for some $n\in \omega$ so that $p(0)=v,p(n+1)=w$ and $(p(j),p(j+1))\in E_{\prec}$ for all $j\leq n$.
\end{obs}

That is there is a $\prec$-monotone finite path from $v$ to $w$ for all $v\prec w\in V(P)$. In particular, if $w$ is the $\prec$-successor of $v$ in $V(P)$ then $v$ and $w$ are connected by an edge. Also, two vertices are connected by a transfinite path if and only if they are connected by a finite path.

Suppose $P$ is a path witnessed by  $\prec$ and  $(V(P),\prec)$ has order type $\kappa$. Then we let $P\uhp \alpha$ denote the unique induced subgraph of $P$ spanned by the initial segment of $(V(P),\prec)$ of order type $\alpha$ (for any $\alpha<\kappa$). Similarly, if $q\in V(P)$ then let $P\uhp q=P\uhp \alpha$ where $\alpha$ is the $\prec$-order type of $\{p\in V(P):p\prec q\}$.

Suppose that $P,Q$ are paths witnessed by $\prec_P$ and $\prec_Q$. We will say that a path $Q$ \emph{end extends} the path $P$ iff $P\subseteq Q$, $\prec_Q\uhp V(P)=\prec_P$ and $v\prec_Q w$ for all $v\in V(P), w\in V(Q)\setm V(P)$. 

If $R,S$ are two paths so that the first point of $S$ has $\prec_R$-cofinally many neighbours in $R$ then $R\cup S$ is a path which end extends $R$ and we denote this path by $R^\frown S$ emphasizing this relation.\\

\subsection*{Edge colourings} An \emph{r-edge colouring} of a graph $\gr$ is a map $c:E\to r$ where $r$ is some cardinal. We write $c(v,w)$ instead of $c(\{v,w\})$ for an edge $\{v,w\}\in E$ for notational simplicity. A \emph{finite edge colouring} is an $r$-edge colouring for some $r\in \mbb N$; throughout this paper $r$ will denote a non zero natural number. We will use the following notation: given an edge colouring $c$ of a graph $\gr$ let $$N_G(v,i)=\{w\in N_G(v):c(v,w)=i\}$$ for $v\in V$  and $$N_G[F,i]=\bigcap \{N_G(v,i):v\in F\}$$ for $F\subseteq V$ and $i\in\ran c$. As we always work with a single colouring one at a time, this notation will lead to no misunderstanding. If we work with a single graph then occasionally we omit the subscript $G$ as well. 

Let us fix an edge colouring $c$ of $G$ with $r$ colours and $i<r$. If $\mc P$ is a graph property (e.g. being a path, being connected...) and $A\subseteq V$ then we say that  $$A \text{\emph{ has property }} \mc P \text{\emph{ in colour }} i$$ with respect to $c$ iff $A$ has property $\mc P$  in the graph $(V,c^{-1}(i))$. In particular, by a \emph{monochromatic path} in $G$ we mean a subgraph $P$ of  $(V,c^{-1}(i))$ which is a path (for some $i<r$).

Throughout the paper, we use standard set theoretic notations consistent with the literature, e.g. \cite{kunen}.

\section{Paths and connectivity} \label{pathconsec}

It is not surprising that notions of \emph{connectivity} are closely related to paths and they will indeed play an important role in our proofs. Let us introduce some terminology:

\begin{definition} Let $\gr$ be a graph, $\kappa$ a cardinal and let $A\subseteq V$. We say that $A$ is \emph{$\sat \kappa$} iff for any $v\neq w\in A$ there are $\kappa$-many finite paths from $v$ to $w$ in $G$ which only intersect in $v$ and $w$. We say that $A$ is \emph{$\kappa$-connected} iff $A$ is $\sat \kappa$ in the graph induced by $A$.
\end{definition}

The name \emph{$\sat \kappa$} sets is justified by the following

\begin{obs}Suppose that $G$ is a graph, $\kappa$ is a cardinal and $A$ is a set of vertices. Then the following are equivalent:
\begin{enumerate}
 \item $A$ is $\sat \kappa$,
\item any $v\neq w\in A$ are not separated by a set of size $<\kappa$ i.e. if $F$ is a set of vertices of size $<\kappa$ not containing $v,w$ then there is a path $P$ in $G$ from $v$ to $w$ which avoids $F$.
\end{enumerate}

\end{obs}

The following is obvious:

\begin{obs}\label{pathconnobs}Every $\omega$-connected countable graph is a path of order type $\omega$. Every countable $\sat \omega$ set is covered by a path of order type $\leq\omega$.
\end{obs}

The next lemma describes a method to find connected subsets of edge coloured graphs and was essentially proved in \cite{conn_decomp}.

\begin{lemma}\label{uftrick} Suppose that $\gr$ is a graph, $A\in [V]^\omega$ and $N_G[F]$ is infinite for all $F\in [A]^{<\omega}$. Given any edge
colouring $c: E \to r$ with $r\in \omega$, there is a partition $d_c: V \to r$ and a colour $i_c < r$
so that
$$N_G[F,i]\cap V_{i_c} \text{ is infinite for all } i < r \text{ and finite set } F\subs
A\cap V_i \text{ where }V_i=d_c^{-1}\{i\}.$$ 

In particular, $A\cap V_i$ is $\lkd$ in colour $i$ for all $i < r$ and if $V=A$ then $V_{i_c}$
is $\omega$-connected as well in colour $i_c$. 
\end{lemma}

The above lemma was used in \cite{monopathI} to deduce Theorem \ref{radostrong} and we will apply it later as well.

The next example shows that Observation \ref{pathconnobs} cannot be extended (word by word) to the uncountable case.

\begin{ex} There is a graph $G$ which contains no uncountable paths however $N_G[F]$ is uncountable for all finite $F\subseteq V(G)$.
\end{ex}
\begin{proof} Take a partition of $\omega_1$ into uncountable sets $X_F$ with $F\in [\omg]^{<\omega}$. Let $G=(\omg,E)$ with $$E=\{\{\alpha,\beta\}:\alpha\in F,\beta\in X_F\setm (\max F+1),F\in [\omg]^{<\omega}\}.$$ It is clear that  $N_G[F]$ is uncountable for all finite $F\subseteq V(G)$ and $|N_G(\alpha)\cap \alpha|<\omega$ for all $\alpha<\omg$.

The following observation leads to a contradiction if $G$ contains an uncountable path.

\begin{clobs}\label{trail}
 If a graph $G=(\omg,E)$ contains a path of size $\omg$ then there is a club $C\subs \omg$ so that for all $\alpha\in C$ there is $\beta\in C\setm \alpha$ with $$\sup (N_G(\beta)\cap \alpha)=\alpha.$$
\end{clobs}

 Indeed, take any countable elementary submodel $M$ of $H(\omega_2)$ with $G,P\in M$ and let $\alpha$ denote the $\prec_P$-minimal element of $P\setm M$. Note that $M\cap P$ is an initial segment of $P$ and $\alpha$ must be a $\prec_P$-limit in $P$. Hence, the infinite set $N(\alpha)\cap \{\xi\in \omg:\xi\prec_P \alpha\}$ is contained in $M\cap \omg \subseteq \alpha$ which finishes the proof. 
\end{proof}

However, every uncountable path contains large unseparable sets:

\begin{obs}\label{satobs}
 If $P$ is a path witnessed by an ordering of type $\omg$ then $\{v\in V(P): |N_P(v)|=\omg\}$ is uncountable and $\sat \omg$ in $P$.
\end{obs}

Finally, we will use elementary submodels in certain proofs to split large, highly connected set of vertices into a sequence of smaller, well-behaved and still fairly connected sets.

\begin{definition}A \emph{nice $\kappa$-chain of elementary submodels for $\mc X$} is a sequence $(M_\alpha)_{\alpha<\cf(\kappa)}$ so that $M_0=\emptyset$, $(M_\alpha)_{1\leq \alpha<\cf(\kappa)}$  is an increasing sequence of elementary submodels of $H(\Theta)$ (for some large enough cardinal $\Theta$) with $|M_\alpha|=\kappa_\alpha$ and
\begin{enumerate}
\item $\mc X\in M_\alpha$ and $M_\alpha\cup \{M_\alpha\}\subseteq M_ \beta$ if $1\leq \alpha<\beta<\cf(\kappa)$,
\item $\kappa_{\alpha+1}\cup\{\kappa_{\alpha+1}\}\subseteq M_{\alpha+1}$ for $\alpha<\cf(\kappa)$,
\item the sequence is continuous i.e. $M_ \beta=\bigcup\{M_\alpha:\alpha<\beta\}$ for any limit $\beta<\cf(\kappa)$,
\item if $\kappa$ is a limit cardinal then $(\kappa_\alpha)_{\alpha<\cf(\kappa)}$ is a strictly increasing sequence of regular cardinals in $\kappa$,
\item if $\kappa$ is regular then $M_\alpha\cap \kappa \in \kappa$, 
\item if $\kappa=\lambda^+$ then $\kappa_\alpha=\lambda$ for all $1\leq \alpha<\kappa$.
\end{enumerate}
\end{definition}

\begin{obs}\label{elem} Let $\gr$ be a graph and $A\subseteq V$ $\kappa$-unseparable. Suppose that  $(M_\alpha)_{\alpha<\cf(\kappa)}$ is a nice $\kappa$-chain of elementary submodels covering $A$ so that $A,G\in M_1$. Then $A\cap (M_{\alpha+1}\setm M_\alpha)$ is $\sat {|M_{\alpha+1}|}$ in $V\cap (M_{\alpha+1}\setm M_\alpha)$ for all $\alpha<\cf(\kappa)$.
\end{obs}
\begin{proof} Fix $\alpha<\cf(\kappa)$ and two vertices $u,v \in A\cap (M_{\alpha+1}\setm M_\alpha)$. As $M_{\alpha+1}\models A$ is $\sat {\kappa_{\alpha+1}^+}$, we can find a ${\kappa_{\alpha+1}^+}$-sequence $(P_\xi)_{\xi<\kappa_{\alpha+1}^+}\in M_{\alpha+1}$ of disjoint finite path from $u$ to $v$. As $M_\alpha\in M_{\alpha+1}$ and $|M_\alpha|<\kappa_{\alpha+1}^+$, we can suppose that each path  $ P_\xi$ is disjoint from $M_\alpha$. Now, using $\kappa_{\alpha+1}\cup \{\kappa_{\alpha+1}\}\subseteq M_{\alpha+1}$, we get that $$\{ P_\xi:\xi<\kappa_{\alpha+1}\}\subseteq M_{\alpha+1}.$$ As each $P_\xi$ is finite, we actually have $P_\xi\subseteq V\cap (M_{\alpha+1}\setm M_\alpha)$ for $\xi<\kappa_{\alpha+1}$ which finishes the proof.
\end{proof}

We refer the reader to \cite{soukel} for an introduction to elementary submodels in combinatorics.

\section{Constructing uncountable paths}\label{sec3line}

Now, we present our most important tools in constructing uncountable paths with Lemma \ref{3line} being the main result of this section. 

\begin{definition} For a path $P$ and $x\prec_P y\in P$ let $P\uhp [x,y)$ denote the segment of $P$ from $x$ to $y$ (excluding $y$) i.e. the graph spanned by $\{z\in V(P):x\leq_P z\prec_P y\}$. For a set $A$ and path $P$ we say that \emph{$P$ is concentrated on $A$} iff $$N(y)\cap A \cap V(P\uhp[x,y))\neq \emptyset$$ for every $\prec_P$-limit $y\in P$ and $x\prec_P y$ in $P$. 
\end{definition}

We will use the following easy observation regularly

\begin{obs} Suppose that $P$ is a path concentrated on a set $A$, $p\in V(P)$ and there is a limit element of $P$ above $p$. Then there is a $q\in A\cap V(P)$ such that $p\prec_P q$ and $P\uhp[p,q)$ is finite. 
\end{obs}

We will apply the next lemma multiple times:

\begin{lemma}\label{lego}
Let $G=(V,E)$ be a graph and $\kappa\geq \omega$. If $A$ is a $\sat \kappa$ subset of $V$ then the following are equivalent:
\begin{enumerate}
	\item there is a path $P$ of order type $\kappa$ concentrated on $A$,
	\item $A$ is covered by the vertices of a path $Q$ of order type $\kappa$ concentrated on $A$.
\end{enumerate}
Moreover, if $a\in A$ and $C\in [A]^{\cf(\kappa)}$ then in clause (2) we can construct $Q$ with first point $a$ and cofinal set $C$. 

\end{lemma}

What this lemma says is that the existence of a large path concentrated on $A$ is sufficient to construct another path which includes all of $A$ and has a prescribed first point and cofinal set. The proof will proceed by taking the path $P$ apart and building a new path $Q$ using that $A$ is a $\sat \kappa$.


\begin{proof}
We prove $(1)\Rightarrow(2)$ by induction on $\kappa$. The result holds for $\kappa=\omega$ by Observation \ref{pathconnobs} so suppose that $\kappa>\omega$ and that we proved for cardinals $<\kappa$. Also, fix $a\in A$, $C\in [A]^{\cf(\kappa)}$ and path $P$ concentrated on $A$; note that we do not need to worry about $C$ if $\kappa$ is regular as every subset of $A$ of size $\kappa$ will be cofinal in $Q$. We distinguish two cases:

\textbf{Case 1:} $\kappa>\cf(\kappa)$.

Let us fix an increasing cofinal sequence of regular cardinals $(\kappa_\alpha)_{\alpha<\cf(\kappa)}$ in $\kappa$ so that $\kappa_0=\cf(\kappa)$ and $\kappa_\beta>\sup\{\kappa_\alpha:\alpha<\beta\}$ for all $\beta<\cf(\kappa)$. 

\begin{lclaim} There are pairwise disjoint paths $\{R_\alpha:\alpha<\cf(\kappa)\}$ in $V\setm (\{a\}\cup C)$ concentrated on $A$ such that 
\begin{enumerate}[(i)]
	\item $R_0$ has order type $\kappa_0$, $R_\alpha$ has order type $\kappa_\alpha+n_\alpha$ for some $n_\alpha\in \omega\setm \{0\}$,
	\item $R_0$ starts with an element of $A$, $R_\alpha$ starts and finishes  with an element of $A$ for $0<\alpha<\cf(\kappa)$,
	\item for every $x,y\in A$ there are $\kappa$ many pairwise disjoint finite paths from $x$ to $y$ in $\{x,y\}\cup V\setm \bigcup_{\alpha<\cf(\kappa)}R_\alpha$.
\end{enumerate}
\end{lclaim}
\begin{proof}We proceed by induction on $\alpha<\cf(\kappa)$. Let $A=\bigcup\{A_\alpha:\alpha<\cf(\kappa)\}$ be an increasing union with $|A_\alpha|\leq \kappa_\alpha$. Simply choose $R_0$ to be a segment of $P$ which satisfies the above conditions (on the starting point and order type). Suppose we constructed $\{R_\alpha:\alpha<\beta\}$ satisfying (i) and (ii) above and sets $\{W_\alpha:\alpha<\beta\}$ so that 
\begin{enumerate}[(a)]
	\item $W_\alpha\in [V]^{\kappa_\alpha}$ and any two points $x\neq y\in A_\alpha$ can be connected by $\kappa_\alpha$ pairwise disjoint finite paths in $\{x,y\}\cup W_\alpha$, and
	\item $W_\alpha\cap R_{\alpha'}=\emptyset$ for $\alpha,\alpha'<\beta$.
\end{enumerate}
Let $X_\beta=\bigcup\{R_\alpha:\alpha<\beta\}\cup \bigcup\{W_\alpha:\alpha<\beta\}$ and note that $X_\beta$ has size less than $\kappa$. As the path $P$ has $\kappa$ many $\kappa_\beta$-limit points, we can select a subpath $R_\beta$ of $P$ (an interval of $P$ starting and finishing  with an element of $A$) of order type $\kappa_\beta+n_\beta$ such that $R_\beta\cap X_\beta=\emptyset$. We can construct now $W_\beta\subs V\setm (X_\beta\cup R_\beta)$ as desired using that $A_\beta$ is $\kappa$-unseparable.
\end{proof}

Fix $\{R_\alpha:\alpha<\cf(\kappa)\}$ as above. Let $C_\alpha$ denote a subset of $A\cap R_\alpha$ which is cofinal in $R_\alpha\uhp \kappa_\alpha$ and let $t_\alpha$ denote the $\kappa_\alpha$-limit point of $R_\alpha$ for $0<\alpha<\cf(\kappa)$. Write $A\setm \bigcup_{\alpha<\cf(\kappa)}R_\alpha$ as $\{A_\alpha:\alpha<\cf(\kappa)\}$ so that $|A_\alpha|\leq \kappa_\alpha$. List $C$ as $\{c_\alpha:\alpha<\cf(\kappa)\}$.

 Construct a sequence of paths $\{Q_\alpha:\alpha<\cf(\kappa)\}$ concentrated on $A$ so that 
\begin{enumerate}
\item $Q_\beta$ end extends $Q_\alpha$ 
for $\alpha<\beta<\cf(\kappa)$,
\item $Q_\alpha$ starts with $a$ and finishes with a point $r_\alpha\in A\cap R_0$,
\item $Q_\alpha\cap R_0\subs R_0\uhp r_\alpha\cup \{r_\alpha\}$ and $Q_\alpha$ covers all points $x\in A\cap R_0$ such that $x<_{R_0}r_\alpha$,
\item $(Q_{\alpha+1}\setm Q_\alpha)\cap C=\{c_\alpha\}$,
\item $Q_\alpha$ covers $A_\alpha \cup (R_\alpha\cap A)$ for $0<\alpha<\cf(\kappa)$,
\item $Q_\alpha\cap R_\beta=\emptyset$ if $\alpha<\beta<\cf(\kappa)$.
\end{enumerate}
Suppose we have $Q_\alpha$ for $\alpha<\beta$. If $\beta=\alpha+1$ then let $r_\beta^-=r_\alpha$, if $\beta$ is limit then let $r_\beta^-$ be the first limit point of $R_0$ above $\{r_\alpha:\alpha<\beta\}$. Note that $r_\beta^-=\sup_{<_{R_0}}\{r_\alpha:\alpha<\beta\}$ if $\beta$ is a limit and hence $Q_{<\beta}=\bigcup\{Q_\alpha:\alpha<\beta\}\cup \{r_\beta^-\}$ is a path concentrated on $A$ by property (3). Let $r^+_\beta <_{R_0} r_\beta\in R_0$ be the first two points of $A$ above $r_\beta^-$ and note that that $R_0\uhp [r_\beta^-,r_\beta]$ is finite. 

\begin{figure}[H]%
\centering

\psscalebox{0.7 0.7} 
{
\begin{pspicture}(0,-2.345)(13.47,2.345)
\psline[linecolor=black, linewidth=0.04](0.8,1.1692857)(9.6,1.1692857)
\psline[linecolor=black, linewidth=0.04, linestyle=dotted, dotsep=0.10583334cm](9.8,1.1692857)(12.2,1.1692857)
\rput[bl](12.9,0.955){\Large{$R_0$}}
\psdots[linecolor=black, dotsize=0.24](0.4,-0.6307143)
\rput[bl](0.0,-0.43071428){\Large{$a$}}
\psbezier[linecolor=black, linewidth=0.04, arrowsize=0.05291666666666667cm 2.0,arrowlength=1.4,arrowinset=0.0]{->}(0.4,-0.6307143)(0.8,-0.6307143)(1.2,-0.23071429)(1.4,0.16928571)(1.6,0.5692857)(1.8,1.7692857)(2.4,1.7692857)(3.0,1.7692857)(3.2,0.3692857)(3.6,0.3692857)(4.0,0.3692857)(4.0,1.0692858)(4.4,1.6692857)(4.8,2.2692857)(5.3142858,1.4407142)(5.9,1.3407143)
\rput[bl](3.1857142,1.855){\Large{$Q_{<\beta}$}}
\psline[linecolor=black, linewidth=0.04](8.2,1.3692857)(8.2,0.9692857)
\psline[linecolor=black, linewidth=0.04](6.4,1.3692857)(6.4,0.9692857)
\psline[linecolor=black, linewidth=0.04](7.4,1.3692857)(7.4,0.9692857)
\rput[bl](6.0,1.5692858){\large{$r_\beta^-$}}
\rput[bl](7.357143,1.5692858){\large{$r_\beta^+,r_\beta\in A$}}
\psline[linecolor=black, linewidth=0.04](4.4,-1.6307143)(8.4,-1.6307143)
\psline[linecolor=black, linewidth=0.04, linestyle=dotted, dotsep=0.10583334cm](8.6,-1.6307143)(10.0,-1.6307143)
\psdots[linecolor=black, dotsize=0.16](10.4,-1.6307143)
\psdots[linecolor=black, dotsize=0.24](11.8,-1.6307143)
\psline[linecolor=black, linewidth=0.04](10.4,-1.6307143)(11.8,-1.6307143)
\rput[bl](12.642858,-1.9021429){\Large{$R_\beta$}}
\psbezier[linecolor=black, linewidth=0.04, linestyle=dashed, dash=0.17638889cm 0.10583334cm, arrowsize=0.05291666666666667cm 2.0,arrowlength=1.4,arrowinset=0.0]{->}(7.4,1.1692857)(7.0,-0.030714286)(6.0,0.16928571)(5.6,0.16928571)(5.2,0.16928571)(4.2,-0.030714286)(4.2,-1.2307143)
\psdots[linecolor=black, dotsize=0.16](4.4,-1.6307143)
\psdots[linecolor=black, dotsize=0.24](7.4,1.1692857)
\psdots[linecolor=black, dotsize=0.24](8.2,1.1692857)
\psbezier[linecolor=black, linewidth=0.04, linestyle=dashed, dash=0.17638889cm 0.10583334cm, arrowsize=0.05291666666666667cm 2.0,arrowlength=1.4,arrowinset=0.0]{->}(11.8,-1.6307143)(11.6,-1.0307143)(11.0,-0.6307143)(10.2,-0.43071428)(9.4,-0.23071429)(8.6,-0.23071429)(8.4,0.5692857)
\psbezier[linecolor=black, linewidth=0.04, linestyle=dashed, dash=0.17638889cm 0.10583334cm, arrowsize=0.05291666666666667cm 2.0,arrowlength=1.4,arrowinset=0.0]{->}(4.4,-1.6307143)(4.4,-2.2307143)(5.1571426,-2.3735714)(5.757143,-2.1735713)(6.357143,-1.9735714)(6.2,-1.2307143)(6.6,-1.2307143)(7.0,-1.2307143)(7.0,-1.6307143)(7.2,-2.0307143)(7.4,-2.4307144)(8.328571,-2.1592858)(8.942857,-2.0164285)
\psdots[linecolor=black, dotsize=0.16](6.2714286,-1.6307143)
\psdots[linecolor=black, dotsize=0.16](7.0285716,-1.6307143)
\psdots[linecolor=black, dotsize=0.24](10.228572,-0.445)
\rput[bl](10.485714,-0.23071429){\large{$c_\beta$}}
\rput[bl](10.3,-2.345){\large{$t_\beta$}}
\end{pspicture}
}

\caption{Extending $Q_{<\beta}$ to $Q_\beta$.}%

\end{figure}
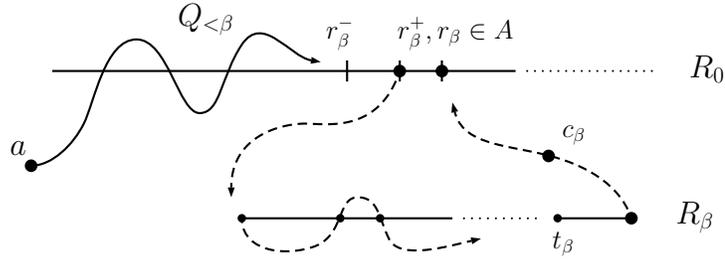

\begin{lclaim} There is a path $S$ concentrated on $A$ in $V\setm (\bigcup\{R_\alpha:\alpha\in \kappa\setm \{\beta\}\}\cup Q_{<\beta})$ such that 
\begin{enumerate}[(i)]
	\item $S$ end extends $R_0\uhp [r_\beta^-,r_\beta^+]$ and the last point of $S$ is $r_\beta$,
	\item $S$ covers $A_\beta\setm Q_{<\beta}\cup (R_\beta\cap A)$,
	\item $S\cap C=\{c_\beta\}$.
\end{enumerate}
\end{lclaim}
\begin{proof} $S$ is constructed using $R_\beta$ and the inductive hypothesis for $\kappa_\beta$. First, let us find a finite path $S'$ with first point $t_\beta$ and the finite end segment of $R_\beta$ so that $S'\cap C=\{c_\beta\}$ and the last point of $S'$ is $r_\beta$. This can be done as $A$ is $\sat \kappa$.

Now, note that $R_\beta\uhp \kappa_\beta$ is a path of order type $\kappa_\beta$ concentrated on $(R_\beta\cap A)\cup (A_\beta\setm Q_{<\beta})$ in $V_\beta=V\setm (\bigcup\{R_\alpha:\alpha\in \kappa\setm \{\beta\}\}\cup Q_{<\beta}\cup C\cup S')$ and that $(R_\beta\cap A)\cup (A_\beta\setm Q_{<\beta})$ is $\sat {\kappa_\beta}$ in $V_\beta$. Hence, we can apply the inductive hypothesis in $V_\beta$ and find a path $S''$ concentrated on $A$ of order type $\kappa_\beta$ with first point $r^+_\beta$, so that $S''$  covers $(R_\beta\cap A)\cup (A_\beta\setm Q_{<\beta})$ and has cofinal set $C_\beta$. We set $S=R_0\uhp [r^-_\beta,r^+_\beta]\smf S''\smf S'$


\end{proof}

Let $Q_\beta=Q_{<\beta}\smf S$ and thus the inductive step is done. Hence the proof for the case when $\kappa$ is singular is finished.\\

\textbf{Case 2:} $\kappa=\cf(\kappa)$.

We fix a nice sequence of elementary submodels $(M_\alpha)_{\alpha<\kappa}$ covering $A$ with $A,G\in M_1$ and let $A_\alpha=M_\alpha\cap A$. Let $p_\alpha=\min_{\prec_P}P\setm M_\alpha$ for $\alpha<\kappa$ and note that $p_\alpha\in M_{\alpha+1}$ and $p_\alpha$ is a $\prec_P$-limit. Also, observe that $$\{p\in A\cap N(p_\beta): p\prec_P p_\beta\}\setm M_\alpha \text{ is infinite }$$ for all $\alpha<\beta<\kappa$; indeed, this follows from the fact that $M_\alpha\cap P$ is a proper initial segment of $P\uhp p_\beta$.


Now, it suffices to construct a sequence of paths $\{Q_\alpha:\alpha<\kappa\}$ concentrated on $A$ so that $Q_1$ has first point $a$ and
\begin{enumerate}

\item $A_\alpha\subs Q_\alpha \subs M_\alpha$,
\item $Q_\alpha\smf (p_\alpha)$ is a path which is an initial segment of $Q_\beta$

\end{enumerate}
for all $\alpha<\beta<\kappa$. Indeed, $\bigcup \{Q_\alpha:\alpha<\kappa\}$ is the path we are looking for.

Suppose we constructed $Q_\alpha$ for $\alpha<\beta$. Let 
\[
 Q_{<\beta} =
  \begin{cases}
	 \bigcup\{Q_\alpha:\alpha<\beta\}\smf (p_\beta) & \text{if } \beta \text{ is a limit,}\\
	  Q_\alpha\smf (p_\alpha) & \text{if } \beta=\alpha+1.

  \end{cases}
\]

 Note that  $Q_{<\beta}$ is a path; for successor $\beta$ this is ensured by (2) while for a limit $\beta$ ensured by (1) and the observation about $p_\beta$ above.

If $\beta$ is a limit, we simply let $Q_\beta=\bigcup\{Q_\alpha:\alpha<\beta\}$; it is easy to see that (1) is satisfied as the chain $(M_\alpha)_{\alpha<\cf(\kappa)}$ is continuous. 

Now suppose $\beta=\alpha+1$. Our goal is to apply the inductive hypothesis and find a path $S$ concentrated on $A$ in $V\cap(M_{\alpha+1}\setm M_\alpha)$ so that 
\begin{enumerate}[(i)]
	\item $S$ has first point $p_\alpha$,
	\item $S$ covers $A\cap M_{\alpha+1}\setm M_\alpha$, and
		\item there is an infinite subset of $N(p_{\alpha+1})\cap A\cap M_{\alpha+1}\setm M_\alpha$ cofinal in $S$.
\end{enumerate}
Indeed, $Q_{\beta}=Q_\alpha \smf S$ will satisfy (1) and (2).

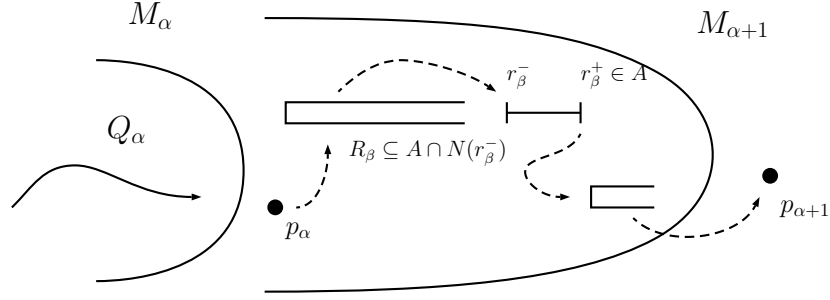
\begin{figure}[H]
\centering

\psscalebox{0.7 0.7} 
{
\begin{pspicture}(0,-2.72)(15.404142,2.72)
\psbezier[linecolor=black, linewidth=0.04](1.6141422,1.7)(2.4141421,1.7)(4.414142,1.4666667)(4.414142,-0.4)(4.414142,-2.2666667)(2.4141421,-2.5)(1.6141422,-2.5)
\psbezier[linecolor=black, linewidth=0.04](4.814142,2.5)(8.699142,2.5)(10.734142,2.3)(12.214142,1.5)(13.694142,0.7)(13.694142,-0.9)(12.214142,-1.7)(10.734142,-2.5)(8.144142,-2.7)(4.814142,-2.7)
\psbezier[linecolor=black, linewidth=0.04, arrowsize=0.05291666666666667cm 2.0,arrowlength=1.4,arrowinset=0.0]{->}(0.014142151,-1.1)(0.41414216,-0.7)(0.6141422,-0.3)(1.2141422,-0.3)(1.8141421,-0.3)(2.4141421,-0.9)(3.6141422,-0.9)
\rput[bl](1.8141421,0.1){\Large{$Q_\alpha$}}
\rput[bl](2.214142,2.3){\Large{$M_\alpha$}}
\rput[bl](13.014142,2.1){\Large{$M_{\alpha+1}$}}
\psdots[linecolor=black, dotsize=0.3](14.414143,-0.5)
\psdots[linecolor=black, dotsize=0.3](5.014142,-1.1)
\rput[bl](5.2141423,-1.7){\large{$p_\alpha$}}
\rput[bl](14.614142,-1.3){\large{$p_{\alpha+1}$}}
\psline[linecolor=black, linewidth=0.04](12.214142,-0.7)(11.014142,-0.7)(11.014142,-1.1)(12.214142,-1.1)
\psline[linecolor=black, linewidth=0.04](8.614142,0.5)(5.2141423,0.5)(5.2141423,0.9)(8.614142,0.9)
\psline[linecolor=black, linewidth=0.04](9.414143,0.7)(10.814142,0.7)
\psline[linecolor=black, linewidth=0.04](9.414143,0.5)(9.414143,0.9)
\psline[linecolor=black, linewidth=0.04](10.814142,0.5)(10.814142,0.9)
\psbezier[linecolor=black, linewidth=0.04, linestyle=dashed, dash=0.17638889cm 0.10583334cm, arrowsize=0.05291666666666667cm 2.0,arrowlength=1.4,arrowinset=0.0]{<-}(9.214142,1.1)(8.614142,1.5)(7.814142,1.7)(7.414142,1.7)(7.014142,1.7)(6.614142,1.5)(6.2141423,1.1)
\psbezier[linecolor=black, linewidth=0.04, linestyle=dashed, dash=0.17638889cm 0.10583334cm, arrowsize=0.05291666666666667cm 2.0,arrowlength=1.4,arrowinset=0.0]{->}(5.414142,-1.1)(5.814142,-1.1)(6.014142,-0.5)(6.014142,0.1)
\rput[bl](9.414143,1.15){$r_\beta^-$}
\rput[bl](10.814142,1.15){$r_\beta^+\in A$}
\psbezier[linecolor=black, linewidth=0.04, linestyle=dashed, dash=0.17638889cm 0.10583334cm, arrowsize=0.05291666666666667cm 2.0,arrowlength=1.4,arrowinset=0.0]{->}(10.814142,0.3)(10.614142,-0.1)(10.014142,-0.1)(9.814142,-0.3)(9.614142,-0.5)(10.014142,-0.9)(10.614142,-0.9)
\psbezier[linecolor=black, linewidth=0.04, linestyle=dashed, dash=0.17638889cm 0.10583334cm, arrowsize=0.05291666666666667cm 2.0,arrowlength=1.4,arrowinset=0.0]{->}(11.814142,-1.3)(12.414143,-1.9)(13.814142,-1.7)(14.214142,-0.9)
\rput[bl](6.414142,-0.3){{$R_\beta\subseteq A\cap N(r_\beta^-)$}}
\end{pspicture}
}
\caption{Extending $Q_{<\beta}$ to $Q_\beta$.}
\end{figure}

Let us pick the cofinal set mentioned in (iii) first: let $R^-\subseteq N(p_{\alpha+1})\cap A\cap M_{\alpha+1}\setm (M_\alpha\cup \{p_\alpha\})$ be infinite and find a path $R$ of order type $\omega$ in  $M_{\alpha+1}\setm M_\alpha$ covering $R^-$ and starting in $R^-$. The path $S$ will end with $R$ and hence property (iii) will be satisfied. Also, $p_\alpha$ might not be in $A$ but a finite segment of $P$  connects $p_\alpha$ to some $q\in A \cap M_{\alpha+1}\setm M_\alpha$.

Now, let $\lambda=|A\cap M_{\alpha+1}\setm M_\alpha|$ and find a point $r_\beta^-\in P\cap M_{\alpha+1}\setm M_\alpha$ which is a $\cf(\lambda)$-limit point of $P$; let $r_\beta^+$ be the first point of $A\cap P$ above $r_\beta^-$. Let $W=V(R\cup P\uhp[p_\alpha,q]\cup P\uhp [r_\beta^-,r_\beta^+])$.

Find a finite path $T$ in $M_{\alpha+1}\setm (M_\alpha\cup W)$ connecting $r_\beta^+$ to the first point of $R$ and let $R'=P\uhp [r_\beta^-,r_\beta^+]\smf T\smf R$. The path $S$ will start with $P\uhp[p_\alpha,q]$ and end with $R'$. Let $W'=W\cup V(T)$.

Finally, pick any $R_\beta\in[N(r_\beta^-)\cap A\cap M_{\alpha+1}\setm (M_\alpha\cup W')]^{\cf(\lambda)}$. Now apply the inductive hypothesis in the graph $G\uhp V\cap M_{\alpha+1}\setm (M_\alpha\cup W')$ for the $\sat \lambda$ set $A \cap M_{\alpha+1}\setm (M_\alpha\cup W')$; we can find a path $S'$ concentrated on $A$ which starts with $P\uhp[p_\alpha,q]$, covers $A \cap M_{\alpha+1}\setm M_\alpha$ and $R_\beta$ is cofinal in $S'$. Note that  $V\cap M_{\alpha+1}\setm M_\alpha$ contains a path which is concentrated on $A \cap M_{\alpha+1}\setm M_\alpha$ and has ordertype $\lambda$; indeed, consider an appropriate segment of the original path $P$ in $M_{\alpha+1}\setm M_\alpha$. 

We are done by letting $S=S'\smf R'$.

\end{proof}

As we will see, there are three main ingredients to constructing a path covering a set $A$ of size $\kappa$ one of which is being $\sat \kappa$.

\begin{definition}Suppose that $\gr$ is graph and $A\subseteq V$. We say that $A$ satisfies $\utakk \kappa $ iff for each $\lambda< \kappa$ there are $\kappa$-many pairwise disjoint paths concentrated on $A$ each of order type $\lambda$.
\end{definition}

If we have a fixed edge colouring we use $\utak \kappa i$ for ``$\utakk \kappa$ in colour $i$" for short. Also, let us mention an easy result for later reference:

\begin{obs}\label{utakeq}Suppose that $\gr$ is a graph, $A\in [V]^\kappa$. Consider the following statements:
\begin{enumerate}
\item there is a path $P$ in $G$ of size $\kappa$ concentrated on $A$,
\item for all $X\in [V]^{<\kappa}$ and $\lambda<\kappa$ there is a path $P$ of order type $\lambda$ disjoint from $X$ which is concentrated on $A$,
	\item $A$ satisfies $\utakk \kappa$.
	
\end{enumerate}
Then $(1)\Rightarrow (2) \Leftrightarrow (3)$.
\end{obs}
\begin{proof}$(1)\Rightarrow (2)$: suppose that $X\in [V]^{<\kappa}$ and $\lambda<\kappa$. If $\kappa$ is regular then $X\cap P$ must be bounded in $P$ and hence an end segment of $P$ is a path disjoint from $X$ which has order type $\kappa$. If $\kappa$ is singular, then $\mu=|X|^+$ is less than $\kappa$ and we repeat the previous argument for $P\uhp \mu$.

(2)$\Rightarrow$(3): suppose that there is $\lambda<\kappa$ and a maximal family $\mc P$ of pairwise disjoint paths concentrated on $A$ of order type $\lambda$ so that $|\mc P|<\kappa$. We can apply (2) to $X=\bigcup \mc P\in [V]^{<{\kappa}}$ to extend $\mc P$ but this contradicts the maximality of $\mc P$. 

(3)$\Rightarrow$(2): suppose that $X\in [V]^{<\kappa}$ and $\lambda<\kappa$. Take a family $\mc P$ of pairwise disjoint paths concentrated on $A$ each of order type $\lambda$ so that $|\mc P|=\kappa$. There is $P\in \mc P$ so that $P\cap X=\emptyset$.

\end{proof}

Clearly, (2) does not imply (1) as $\utakk \kappa$ is easily satisfied in a graph which has no connected component of size $\kappa$. 

The next lemma will be our main tool in constructing paths.

\begin{lemma}\label{3line} Suppose that $\gr$ is a graph, $\kappa\geq \omega$, $A\in [V]^\kappa$ and
\begin{enumerate}
\item $A$ is $\sat \kappa$, and

\end{enumerate}
if $\kappa>\omega$ then
\begin{enumerate}
\setcounter{enumi}{1}

\item $A$ satisfies $\utakk \kappa$, and
\item there is a nice sequence of elementary submodels  $(M_\alpha)_{\alpha<\cf(\kappa)}$ for $\{A,G\}$ covering $A$ so that there is $x_\beta\in A\setm M_\beta,y_\beta\in V\setm M_\beta$ with $\{x_\beta,y_\beta\}\in E$ and $$|N_G(y_\beta)\cap A\cap M_\beta\setm M_{\alpha}|\geq \omega$$ for all $\alpha<\beta<\cf(\kappa)$.
\end{enumerate}
Then $A$ is covered by a path $P$ concentrated on $A$.
\end{lemma}

Note that condition (3) only makes sense for $\kappa>\omega$; indeed, if $A$ is countable and $A\in M_1$ then $A\subseteq M_1$ as well. However, every countably infinite $\sat\omega$ set $A$ is covered by a path of order type $\omega$.

\begin{proof} If $\kappa=\omega$ then we can apply Observation \ref{pathconnobs} to finish the proof. Suppose $\kappa>\omega$.

Let us fix $(M_\alpha)_{\alpha<\cf(\kappa)}$ and $\{x_\alpha,y_\alpha:\alpha<\cf(\kappa)\}$ as in clause (3); we can suppose that $x_\alpha,y_\alpha\in M_{\alpha+1}$. Let $A_\alpha=M_\alpha\cap A$ for $\alpha<\cf(\kappa)$. It suffices to construct a sequence of paths $\{P_\alpha:\alpha<\cf(\kappa)\}$ concentrated on $A$ so that
\begin{enumerate}[(i)]
\item $P_\beta$ end extends $P_\alpha$,
\item $A_\alpha \subseteq P_\alpha\subseteq M_\alpha$,
\item $N_G(y_{\alpha+1})\cap A\cap M_{\alpha+1}\setm M_{\alpha}$ is cofinal in $P_{\alpha+1}$
\end{enumerate}
for all $\alpha<\beta<\cf(\kappa)$. We set $P=\bigcup\{P_\alpha:\alpha<\cf(\kappa)\}$ which will finish the proof.

Suppose we constructed $P_\alpha$ for $\alpha<\beta$ as above; if $\beta$ is a limit ordinal then we set $P_\beta=\bigcup\{P_\alpha:\alpha<\beta\}$. Suppose that $\beta=\alpha +1$; note that $P_\alpha\smf (y_\alpha,x_\alpha)$ is still a path regardless whether $\alpha$ is a limit or successor by (ii) and (iii). It suffices to find a path $S\subs M_{\alpha+1}\setm M_\alpha$ concentrated on $A$ starting at $x_\alpha$ so that  $N(y_{\alpha+1})\cap A\cap M_{\alpha+1}\setm M_{\alpha}$ is cofinal in $S$ and $A_{\alpha+1}\setm A_\alpha\subs S$; indeed, we set $P_{\alpha+1}=P_\alpha\smf (y_\alpha)\smf S$  which finishes the proof. 

We will essentially repeat the proof of Lemma \ref{lego} in the regular case. Recall that $N(y_{\alpha+1})\cap A\cap M_{\alpha+1}\setm M_{\alpha}$ is infinite. First, find a path $R$ of order type $\omega$  in $M_{\alpha+1}\setm (M_{\alpha}\cup \{x_\alpha,y_\alpha\})$ so that $$|R\cap N_G(y_{\alpha+1})\cap A\cap M_{\alpha+1}\setm M_{\alpha}|\geq \omega$$ and $R$ starts at a point $r$ so that $|N_G(r)\cap A_{\alpha+1}\setm A_\alpha|\geq \cf(\nu)$ where $\nu=|A_{\alpha+1}\setm A_\alpha|$. 
The only difficulty here is to find such an $r$; if $\kappa$ is limit we can use $\utakk \kappa$ to find a path $Q\in M_{\alpha+1}$ concentrated on $A$ of size $|M_{\alpha+1}|^+$ and $r$ can be chosen to be the first $|\cf(\nu)|$-limit of $Q$ (note that $\nu<|M_{\alpha+1}|^+$). A finite segment of $Q$ connects $r$ to some $r'\in Q \cap A$ and we continue to construct $R$ from this finite path. If $\kappa$ is a successor then we must have $\kappa=\nu^+$ (by the definition of a nice sequence of models)  and note that $|N_G(y_{\alpha+\nu})\cap A\setm A_\alpha|\geq \cf(\nu)$ for any $\alpha<\kappa$. Reflecting this property into $M_{\alpha+1}$ we find $y,x\in M_{\alpha+1}\setm (M_{\alpha}\cup \{x_\alpha,y_\alpha\})$ so that $\{y,x\}\in E$, $x\in A$ and $|N(y)\cap A_{\alpha+1}\setm A_\alpha|\geq \cf(\nu)$. We can start $R$ by $y$ and $x$ and connect the rest of the points using that $A$ is $\sat \kappa$ in $M_{\alpha+1}\setm M_\alpha$.

Now, we construct $S$ with the above required properties so that it has order type  $\nu+\omega$ and $R$ is the last $\omega$ many points of $S$. Indeed, $\utakk \kappa$ implies that $G\uhp (V\cap M_{\alpha+1}\setm (M_\alpha\cup R\cup \{y_\alpha\}))$ contains a path of order type $\nu$ concentrated on $A_{\alpha+1}\setm A_\alpha$ so by applying Lemma \ref{lego} we can find a path $S'$  starting at $x_\alpha$, concentrated on $A$ and of order type  $\nu$ which covers $A_{\alpha+1}\setm (A_\alpha\cup R)$ while $N_G(r)\cap A$ is cofinal in $S'$; we set $S=S'\smf R$ which finishes the proof.

\end{proof}

\section{The existence of monochromatic paths}\label{secdich}

Our goal in this section is to find large monochromatic paths in certain edge coloured graphs $G$ by finding a set $A\subseteq V(G)$ which satisfies all three conditions of Lemma \ref{3line} in a colour.

\subsection{Preparations} As we stated in the introduction, we aim to deal with certain graphs which are almost complete:

\begin{definition} We call a graph $G=(V,E)$ \emph{$\kappa$-complete} iff $|V|\geq \kappa$ and $$|V\setm N_G(x)|<\kappa$$ for all $x\in V$.
\end{definition}

Let us start with some basic observations.

\begin{obs}\label{complobs} \begin{enumerate}
	\item Any $\kappa$-complete graph $\gr$ is $|V|$-complete.
	\item If $\gr$ is $\kappa$-complete then any subset $X\in[V]^\kappa$ spans a $\kappa$-complete subgraph.
\end{enumerate}
\end{obs}
\begin{proof}(1) If $G$ is $\kappa$-complete then $|V|\geq \kappa$ and hence $|V\setm N_G(x)|<\kappa\leq |V|$ for all $x\in V$. Thus $G$ is $|V|$-complete.

(2) If $X\subseteq V$ has size $\kappa$ then $|X\setm N_G(x)|\leq |V\setm N_G(x)|<\kappa$.
\end{proof}

To prove our decomposition result about $\kappa$-complete graphs, we need to look at edge-colourings of certain large bipartite graphs. Let $\kappa$ be a cardinal. We let $K_{\kappa,\kappa}$ denote the complete bipartite graph with classes of size $\kappa$.

 We let $H_{\kappa,\kappa}$ denote the graph  $(\kappa\times \{0\}\cup \kappa\times \{1\},E)$ where $$\{(\alpha,i),(\beta,j)\}\in E \iff i=0,j=1 \text{ and } \alpha<\beta\in\kappa.$$

$H_{\kappa,\kappa}$ is a bipartite graph and we call the set of vertices $\kappa\times\{0\}$ in $H_{\kappa,\kappa}$  \emph{the main class of $H_{\kappa,\kappa}$}. If $H$ denotes a copy of $\halff \kappa$ then let $H\uhp \alpha$ stand for $\halff \kappa\uhp \alpha\times 2$ for any $\alpha<\kappa$.\\

In order to carry out our proofs we need to introduce a class of graphs closely related to the graph $\halff \kappa$.

\begin{definition} We say that a graph $G=(V,E)$ is of \emph{type $\halff \kappa$} iff $V=A\cup B$ where $A=\{a_\xi:\xi<\kappa\},B=\{b_\xi:\xi<\kappa\}$ are 1-1 enumerations and $$\{a,b\}\in E(G) \text{ if } a=a_\xi, b=b_\zeta \text{ for some }\xi\leq \zeta<\kappa.$$ 
\end{definition}

We will call $A$ the \emph{main class} of $G$ and $(A,B)$ with the inherited ordering is the $\halff \kappa$-decomposition of $G$. As before, we use the notation $G\uhp \lambda$ to denote $G\uhp\{a_\xi,b_\xi:\xi<\lambda\}$.

We will mainly apply this definition in two cases: when the classes $A$ and $B$ of the graph $G$ of type $\halff \kappa$ are disjoint (i.e. $G$ is isomorphic to \emph{the graph $\halff \kappa$}) and when the main class equals $V(G)$.

\begin{obs}\label{Hbuild} Suppose that $\gr$ is a graph, $A\in [V]^\kappa$ for some cardinal $\kappa$ and $A$ is the increasing union of sets $\{A_\alpha:\alpha<\cf(\kappa)\}$ where $|A_\alpha|<\kappa$ and $$|N[A_\alpha]|=\kappa$$ for all $\alpha<\cf(\kappa)$. Then there is a subgraph $H$ of $G$ of type $\halff \kappa$ with main class $A$. 
\end{obs}
\begin{proof} Find an enumeration $A=\{a_\xi:\xi<\kappa\}$ so that for every $\zeta<\kappa$ there is $\alpha_\zeta<\cf(\kappa)$ with $\{a_\xi:\xi\leq\zeta\}\subseteq A_{\alpha_\zeta}$. Hence $$|N[\{a_\xi:\xi\leq\zeta\}]|\geq |N[A_{\alpha_\zeta}]|=\kappa.$$

Now, we can inductively find vertices $b_\zeta\in N[\{a_\xi:\xi\leq \zeta\}]\setm\{b_\xi:\xi<\zeta\}$ for all $\zeta<\kappa$ and hence $A\cup \{b_\xi:\xi<\kappa\}$ is the type $\halff \kappa$ subgraph.
\end{proof}

\begin{obs}\label{comptype} Suppose that $G=(V,E)$ is of type $\halff \kappa$ with main class $V$. Then there is a $\kappa$-complete graph embedded in $G$.

Conversely, if $\gr$ is a $\kappa$-complete graph of size $\kappa$ then $G$ is of type $\halff \kappa$ with main class $V$.
\end{obs}
\begin{proof} If $(A,B)$ is the $\halff\kappa $ decomposition of $G$ then we have $B\subseteq A=V$ and $G\uhp B$ is the $\kappa$-complete subgraph.

Second, suppose that $G$ is $\kappa$-complete and list $V$ in type $\kappa$ as $A=\{a_\xi:\xi<\kappa\}$. If $\kappa$ is regular then $N[\{a_\xi:\xi<\zeta\}]$ has size $\kappa$ for all $\zeta<\kappa$ hence Observation \ref{Hbuild} finishes the proof.

If $\kappa$ is singular then let us take an increasing, continuous and cofinal sequence $\{\kappa_\alpha:\alpha<\cf(\kappa)\}$ in $\kappa$ and let $A_\alpha=\{a_\xi:\xi<\kappa_\alpha, |V\setm N(a_\xi)|<\kappa_\alpha\}\in [V]^{<\kappa}$. Note that $V=\bigcup_{\alpha<\cf(\kappa)}A_\alpha$ is an increasing union of sets of size $<\kappa$ and $N[A_\alpha]$ has size $\kappa$. Again, Observation \ref{Hbuild} can be applied which finishes the proof.
\end{proof}

 \begin{obs}\label{onecol} Suppose that $H$ is of type $\halff \kappa$ for some cardinal $\kappa$. Then there is a path of order type $\kappa$ which covers and is concentrated on the main class of $H$.
\end{obs}
This result is trivial for the graph $\halff \kappa$, however we have to be somewhat cautious when the two classes of $H$ intersect.

\begin{proof} Let $A=\{a_\xi:\xi<\kappa\},B=\{b_\xi:\xi<\kappa\}$ witness that $H$ is of type $\halff \kappa$. We define an increasing sequence of paths $\{P_\alpha:\alpha\in D\}$ where $D=\{\alpha<\kappa: \alpha \text{ is a limit of limits}\}$ by induction on $\alpha$ such that
\begin{enumerate}
	\item $P_\alpha$ is a path concentrated on $A$,
	\item $P_\alpha\cap A$ is a cofinal subset of $P_\alpha$,
	\item $P_\alpha\subseteq H\uhp \alpha+\omega+\omega$, and
	\item $P_\alpha$ covers $a_\alpha$
\end{enumerate} for all $\alpha\in D$.

Let us define $P_0$ inductively as $(p^0_n:n\in\omega)$ where 
\[
 p^0_n =
  \begin{cases}
	 a_{l_n} & \text{where } l_n=\min\{l\in\omega:a_{l}\notin\{p^0_m:m<n\}\} \text{ if } n \text{ is even,}\\
	  b_{\omega+k_n} & \text{where } k_n=\min\{k\in\omega:b_{\omega+k}\notin\{p^0_m:m<n\}\} \text{ if } n \text{ is odd.}

  \end{cases}
\]

Suppose that $P_\alpha$ is defined for $\alpha<\beta$ where $\beta\in D$ and let $P_{<\beta}=\bigcup\{P_\alpha:\alpha\in \beta\cap D\}$. Let $$\delta=\min\{\zeta\in\kappa: (P_{<\beta}\cup\{a_\beta\})\subseteq H\uhp \zeta\}.$$ 

It is easy to see that $\delta\leq \beta$.
Observe that $P_{<\beta}\smf (b_{\delta+\omega})$ is a path concentrated on $A$. Let
\[
 P^-_\beta =
  \begin{cases}
	 P_{<\beta}\smf (b_{\delta+\omega}) & \text{if } a_\beta\in  P_{<\beta}\smf (b_{\delta+\omega})\\
	   P_{<\beta}\smf (b_{\delta+\omega},a_\beta,b_{\delta+\omega+1}) & \text{if }  a_\beta\notin  P_{<\beta}\smf (b_{\delta+\omega}).
		
   \end{cases}
\]

By induction on $n<\omega$, define
\[
 p^\beta_n =
  \begin{cases}
	 
	   a_{\delta+l_n} & \text{where } l_n=\min\{l\in\omega:a_{\delta+l}\notin P^-_\beta\cup\{p^\beta_m:m<n\}\} \text{ if } n \text{ is even,}\\
		 b_{\delta+\omega+k_n} & \text{where } k_n=\min\{k\in\omega:b_{\delta+\omega+k}\notin P^-_\beta\cup\{p^\beta_m:m<n\}\} \text{ if } n \text{ is odd.}
		
   \end{cases}
\]
We let $$P_\beta=P^-_\beta\smf(p^\beta_n)_{n\in\omega}.$$ Finally, set $P=\bigcup\{P_\alpha:\alpha\in D\}$ and note that $P$ is a path concentrated on $A$ which also covers $A$.

\end{proof}

Let $(\ih)_{\kappa,r}$ denote the statement that 

\begin{center}

\begin{framed}
  for any $r$-edge colouring of a graph $G$ of type $\halff \kappa$ with main class $A$, there is an $X\subseteq A$ of size $\kappa$ and $i<r$ so that $X$ satisfies all three conditions of Lemma \ref{3line} in colour $i$. 
\end{framed}

\end{center}

Let $(\ih)_{\kappa}$ denote 
\begin{center}
	\begin{framed}
	$(\ih)_{\kappa,r}$ holds for all $r\in\omega$.
\end{framed}
\end{center}

Note that in Lemma \ref{uftrick} we showed that $(\ih)_{\omega}$ holds. Furthermore:
 
\begin{obs}\label{baseobs} For any graph $G$ of type $\halff \kappa$, the main class of $G$ satisfies all three conditions of Lemma \ref{3line}. In particular, $(\ih)_{\kappa,1}$ holds for all $\kappa$.
\end{obs}
\begin{proof} Fix $G$ of type $\halff \kappa$ with main class $A=\{a_\xi:\xi<\kappa\}$  and second class $B=\{b_\xi:\xi<\kappa\}$; we suppose $\kappa>\omega$. $A$ is clearly $\sat \kappa$ and $\utakk \kappa$ is satisfied by Observation \ref{onecol} and Observation \ref{utakeq}. Now, for the third property take any nice sequence of  elementary submodels  $(M_\alpha)_{\alpha<\cf(\kappa)}$ covering $A$ with $A,G\in M_1$ and suppose that the enumeration $\{a_\xi:\xi<\kappa\}$ is also in $M_1$. Let $x_\beta=a_{\xi_\beta},y_\beta=b_{\xi_\beta}$ where $\xi_\beta=\min(\kappa\setm M_\beta)$. 

Now $\{x_\beta,y_\beta\}\in E$, $x_\beta,y_\beta\in M_{\beta+1}\setm M_\beta$ and we need to show that $$|N(y_\beta)\cap A\cap M_\beta\setm M_\alpha|\geq \omega$$ for all $\alpha<\beta$. Fix $\alpha<\beta$ and look at $\xi_\alpha=\min (\kappa\setm M_\alpha)$. As $\xi_\alpha<\xi_\alpha+\omega<\xi_\beta$, we get that $$\{a_{\xi_\alpha+i}:i<\omega\}\subseteq N(y_\beta)\cap A\cap M_\beta\setm M_\alpha.$$

\end{proof}

From now on in this section, we work towards showing that  $(\ih)_{\kappa}$ holds for all $\kappa$. Note that once $(\ih)_\kappa$ is proved, Lemma \ref{3line} implies that every finite-edge coloured graph $G$ of type $\halff\kappa$ contains a monochromatic path of size $\kappa$.

\subsection{The first main step}

We wish to determine if a subset $A$ of an edge coloured graph satisfies condition (3) of Lemma \ref{3line} in a given colour.

\begin{lemma}\label{ml1}  Let $\kappa$ be an uncountable cardinal. Suppose that $c$ is an $r$-edge colouring of a graph $G=(V,E)$ of type $\halff\kappa$ with $\halff \kappa$-decomposition $(A,B)$. If $i<r$ then either

\begin{enumerate}[(a)]
 \item $A$ satisfies condition (3) of Lemma \ref{3line} in colour $i$, or
\item there is $\tilde A\in [A]^{<\kappa}$ so that $A\setm \tilde A$ is covered by a graph $H$ of type $\halff \kappa$ with main class $A\setm \tilde A$ so that $i\notin \ran(c\uhp E(H))$.
\end{enumerate}
\end{lemma}


We need the following claims:

\begin{lclaim}\label{largecofclaim}
Suppose that $\kappa\geq \cf(\kappa)=\mu>\omega$, $c$ is an $r$-edge colouring of a graph $G=(V,E)$ of type $\halff\kappa$ with $\halff \kappa$-decomposition $(A,B)$. Suppose that $\{M_\alpha:\alpha<\mu\}$ is a nice $\kappa$-chain of elementary submodels covering $V$ with $G,A,B,c\in M_1$. If $i< r$ then either
\begin{enumerate}[(a)]
 \item there is a club $ C \subseteq \mu$ so that for every $\beta \in C$ there is $x_\beta\in A\setm M_\beta,y_\beta\in B\setm M_\beta$ such that $c(x_\beta,y_\beta)=i$ and $$|N(y_\beta,i)\cap A\cap M_\beta\setm M_{\alpha}|\geq \omega$$ for all $\alpha<\beta$, or
\item there is $\tilde A\in [A]^{<\kappa}$ so that $A\setm \tilde A$ is covered by a graph $H$ of type $\halff \kappa$ with main class $A\setm \tilde A$ so that $i\notin \ran(c\uhp E(H))$.
\end{enumerate}
\end{lclaim}

\begin{proof}
 Suppose that (a) fails i.e. there is a stationary set $S\subs \mu$ so that for all $\beta\in S$ and $x\in A\setm M_\beta,y\in B\setm M_\beta$ with $c(x,y)=i$ we have $$|N(y,i)\cap A\cap M_\beta\setm M_{\alpha}|<\omega$$ for some $\alpha<\beta$.

 Let $G_{\neq i}=(V,c^{-1}(r\setm \{i\}))$. Note that 
\begin{clobs}\label{obss}
If there is an $\alpha\in S$ and $\lambda<\kappa$ so that $$|B\cap N(x,i)| \leq \lambda$$ for every $x\in A\setm M_\alpha$ then (b) holds with $\tilde A=A\cap M_\alpha$.
 \end{clobs}

Indeed, we can apply Observation \ref{Hbuild} to $A\setm \tilde A$ in the graph $G_{\neq i}$ for the initial segments of the $\halff\kappa$ ordering.

Otherwise, we distinguish two cases:

\textbf{Case 1:} $\kappa$ is regular. Recall that we have $M_\alpha\cap \kappa\in \kappa$ and hence $x\in A\cap M_\alpha,y\in B\setm M_\alpha$ implies  $\{x,y\}\in E$.

Select $x_\beta \in A\setm M_\beta$ and $y_\beta\in B\setm M_\beta$ with $c(x_\beta,y_\beta)=i$; this can be done by Observation \ref{obss}.  Hence $$|N(y_\beta,i)\cap A\cap M_\beta\setm M_{\alpha}|< \omega$$ for some $\alpha<\beta$. That is, there is $\alpha(\beta)<\beta$ so that 
$$N(y_\beta,i)\cap A\cap M_\beta\subs M_{\alpha(\beta)}$$ for all $\beta\in S\cap \lim(\mu)$ (where $\lim(\mu)$ denotes the set of limit ordinals in $\mu$).

Apply Fodor's pressing down lemma to the regressive function $\beta\mapsto \alpha(\beta)$ on the stationary set $S\cap \lim(\kappa)$ and find stationary $T\subseteq S\cap \lim(\kappa)$ and $\tilde \alpha\in \kappa$ so that $\alpha(\beta)=\tilde \alpha$ for all $\beta\in T$. It is easy to see that (b) is satisfied with $\tilde A=A\cap M_{\tilde \alpha}$. Indeed, if $x\in A_\alpha=A\cap M_\alpha\setm \tilde A$ and $\beta\in T\setm \alpha$ then $\{x,y_\beta\}\in E$ and $c(x,y_\beta)\neq i$ (for any $\alpha\in \kappa\setm \tilde \alpha$). In turn $$|N_{G_{\neq i}}[A_\alpha]|\geq |\{y_\beta:\beta\in T\setm \alpha\}|\geq \kappa.$$ Hence we can apply Observation \ref{Hbuild} to $A\setm \tilde A$ in $G_{\neq i}$.

\textbf{Case 2:} $\kappa$ is singular. Recall that $\kappa_\alpha=|M_\alpha|$ is a strictly increasing cofinal sequence of cardinals in $\kappa\setm \cf(\kappa)$. Select $x_\beta \in A\setm M_\beta$  and find $Y_\beta\in [B\setm M_\beta]^{\kappa_\beta^+}$ so that $Y_\beta\subseteq N(x_\beta,i)$ for each $\beta\in S$; this can be done by Observation \ref{obss}. We can suppose, by shrinking $Y_\beta$, that there is a finite set $F_\beta$ and $\alpha(\beta)<\beta$ with $$F_\beta=N(y,i)\cap A\cap M_\beta\setm M_{\alpha(\beta)}$$ for all $y\in Y_\beta$. The importance here is that $\alpha(\beta)$ and $N(y,i)\cap A\cap M_\beta\setm M_{\alpha(\beta)}$ does not depend on $y\in Y_\beta$ which can be done as there are only $|\beta|$ choices for $\alpha(\beta)$ and $\kappa_\beta$ choices for $F_\beta$ while $\kappa_\beta^+$ choices for $y\in Y_\beta$. 

Apply Fodor's pressing down lemma to the regressive function $\beta\mapsto \alpha(\beta)$ and find a stationary $T\subseteq S$ and $\tilde \alpha\in \cf(\kappa)$ so that $\alpha(\beta)=\tilde \alpha$ for all $\beta\in T$. Let $\tilde A=(A\cap M_{\tilde \alpha})\cup \bigcup\{F_\beta:\beta\in T\}$ and note that $|\tilde A|<\kappa$. 

As before, if $x\in A_\alpha=A\cap M_\alpha\setm \tilde A$ and $\beta\in T\setm \alpha$ then $\{x,y\}\in E$ and $c(x,y)\neq i$ for any  $y\in Y_\beta$ and $\alpha\in \kappa\setm \tilde \alpha$. In turn $$|N_{G_{\neq i}}[A_\alpha]|\geq |\bigcup\{Y_\beta:\beta\in T\setm \alpha\}|\geq \kappa.$$ Hence we can apply Observation \ref{Hbuild} to $A\setm \tilde A$ in $G_{\neq i}$.


\end{proof}

\begin{lclaim}\label{ctblesingclaim}
 Suppose that $\kappa>\omega=\cf(\kappa)$ and $c$ is an $r$-edge colouring of a graph $\gr$ of type $\halff\kappa$ with $\halff \kappa$-decomposition $(A,B)$. Suppose that $\{M_n:n\in\omega\}$ is a nice $\kappa$-chain of elementary submodels covering $V$ with $G,A,B,c\in M_1$. If $i<r$ then either
\begin{enumerate}[(a)]
 \item there is an increasing sequence $\{n_k:k\in\omega\}\subseteq \omega$ with $n_0=0$ such that for all $k<\omega$ there is $x_k\in A\setm M_{n_{k+1}},y_k\in B\setm M_{n_{k+1}}$ with $c(x_k,y_k)=i$ and $$|N(y_k,i)\cap A\cap M_{n_{k+1}}\setm M_{n_k}|\geq \omega,$$ or  
\item there is $\tilde A\in [A]^{<\kappa}$ so that $A\setm \tilde A$ is covered by a graph $H$ of type $\halff \kappa$ with main class $A\setm \tilde A$ so that $i\notin \ran(c\uhp E(H))$.
\end{enumerate}
\end{lclaim}

\begin{proof}
 Suppose that (a) fails; hence there is an $\ell\in \omega$ such that for every $x\in A\setm M_n, y\in B\setm M_n$ with $c(x,y)=i$ we have $$|N(y,i)\cap A\cap M_n\setm M_l|<\omega$$ for all $n\in\omega\setm \ell$.
\begin{clobs}
 If there is $n\in\omega$ and $\lambda<\kappa$ so that $N(x,i)\leq\lambda$ for all $x\in A\setm M_n$ then (b) holds with $\tilde A=A\cap M_n$.
\end{clobs}

Indeed, we can apply Observation \ref{Hbuild} to $A\setm \tilde A$ in the graph $G_{\neq i}=(V,c^{-1}(r\setm \{i\}))$.

Otherwise, we can select $x\in A\setm M_n$ and $Y_n\in [M_{n+1}\setm M_\ell]^{|M_n|^+}$ with $Y_n\subs N(x_n,i)$ for all $n\in\omega\setm \ell$. We can suppose, by shrinking $Y_n$, that there is a finite $a_n\subs A\cap M_n$ so that $$N(y,i)\cap A\cap M_n\setm M_\ell=a_n$$ for all $y\in Y_n$. Let $\tilde A=( A\cap M_\ell)\cup \bigcup\{a_n:n\in\omega\setm \ell\}$. As before, in Case 2 of the proof of Claim \ref{largecofclaim}, applying Observation \ref{Hbuild} to $A\setm \tilde A$ in $G_{\neq i}$ finishes the proof.

\end{proof}

Hence, we arrived at the
 
\begin{proof}[Proof of Lemma \ref{ml1}] Assume that (b) fails in Lemma \ref{ml1}. Hence condition (b) of Claim \ref{ctblesingclaim} (if $\cf(\kappa)=\omega$) or \ref{largecofclaim} (if $\cf(\kappa)>\omega$)  fails for colour $i$. In turn, we have a nice sequence of elementary submodels satisfying condition (3) of Lemma \ref{3line} in colour $i$ by condition (a) of Claim \ref{ctblesingclaim} or \ref{largecofclaim} (respectively).

\end{proof}

\subsection{The second main step} \label{ml2sec}

Now, we would like to determine if, in an edge coloured graph of type $\halff \kappa$, a $\sat \kappa$ subset satisfies $\utakk \kappa$ in some colour.

\begin{lemma}\label{ml2}Let $\kappa$ be an infinite cardinal. Suppose that $c$ is an $r$-edge colouring of a graph $G=(V,E)$ of type $\halff\kappa$ with $\halff \kappa$-decomposition $(A,B)$. Let $I\in [r]^{<r}$, $X\in [A]^\kappa$ and suppose that  $X$ is $\sat \kappa$ in all colours $i\in I$. If $(\ih)_{\lambda}$ holds for $\lambda<\kappa$ then either

\begin{enumerate}[(a)]
 \item there is an $i\in I$ such that $X$ satisfies  $\utak \kappa i$, or
\item there is $ \tilde X\in [X]^{<\kappa}$ and a partition $\{X_j:j\in r\setm I\}$ of $X\setm \tilde X$ such that $$|N(x,j)\cap N(x',j)\cap B|=\kappa$$ for all $x,x'\in X_j$ and $j\in r\setm I$.
\end{enumerate}
 In particular,  the sets $X_j$ given by condition (b) are $\sat \kappa$ in colour $j$ in $X_j$.

Moreover, if $B\subs X$ then there is $j\in I\setm r$ so that $X_j$ is $\kappa$-connected in colour $j$.
\end{lemma}

The proof of Lemma \ref{ml2} (at the end of Section \ref{ml2sec}) will be achieved through a series of claims below. The main application of Lemma \ref{ml2} is in the proof of Theorem \ref{halff}.


\begin{ldefinition} Suppose that $\lambda$ is a cardinal, $\gr$ is graph with an $r$-edge colouring $c$. A \emph{$\lambda$-configuration in colours $I\subseteq r$} is a pairwise disjoint family $\mc X=\{a_\xi:\xi<\lambda\}\subs [V]^{<\omega}$ and points $\mc Y=\{y_\xi:\xi<\lambda\}$ such that $$y_\zeta\in \bigcup\{N(x,i):x\in a_\xi,i\in I\}$$ for all $\xi\leq \zeta <\lambda$.
\end{ldefinition}


\begin{lclaim} \label{clmA} Suppose that $\lambda$ is a cardinal, $\gr$ is graph with an $r$-edge colouring $c$. Let $\mc X, \mc Y$ be a $\lambda$-configuration in colours $I\subseteq r$. Suppose that for each $i\in I$ there is $Y_i\subseteq V$ so that $\bigcup \mc X$ is $\sat \lambda$ in colour $i$ inside $V_i=\bigcup \mc X\cup Y_i$.

 Then $(\ih)_{\lambda,|I|}$ implies that there is an $i\in I$ and a path $P$ in colour $i$ concentrated on $\bigcup \mc X$ which is inside $V_i$ and has order type $\lambda$.
\end{lclaim}

\begin{proof}
Let $\mc X=\{a_\xi:\xi<\lambda\}$ and $\mc Y=\{y_\xi:\xi<\lambda\}$ denote the $\lambda$-configuration. By setting $a_\xi'=\bigcup \{a_{\xi+i}:i<|I|+1\}$ and $y_\xi'=y_{\xi+|I|+1}$ for $\xi<\lambda$ limit we get that for all limit ordinals $\xi\leq \zeta <\lambda$ there is an $i\in I$ so that $$|\{x\in a_\xi':c(x,y_\zeta')=i\}|\geq 2.$$ As $\{a'_\xi:\xi<\lambda \text{ limit}\},\{y'_\xi:\xi<\lambda \text{ limit}\}$ is also a $\lambda$-configuration in colours $I$, we will suppose that the original $\lambda$-configuration had this property already.

Also, by thinning out, we can suppose that $(\bigcup \mc X)\cap \mc Y=\emptyset$ and for all $i\in I$, $\xi <\lambda$ and $x,x'\in a_\xi$ there are $\lambda$ many disjoint finite $i$-monochromatic paths in $V_i$ from $x$ to $x'$ which avoid $\mc Y$ and all other points of $\bigcup \mc X$.

Define a colouring of the graph $\halff \lambda$ by $$d((\xi,0), (\zeta,1))=i \text{ iff } |\{x\in a_\xi:c(x,y_\zeta)=i\}|\geq 2$$ and $i$ is minimal such. Note that $d$ is well defined by our previous preparation. Now $(\ih)_{\lambda,|I|}$ implies that there is a path $Q$ of colour $i$ and size $\lambda$ concentrated on the main class of $\halff \lambda$ for some $i\in I$.

\begin{sclaim}There is a path $P$ of colour $i$ and order type $\lambda$ in $G\uhp V_i$ concentrated on $\bigcup \mc X$.
\end{sclaim}
\begin{proof}
Let $Q=\{q_\nu:\nu<\lambda\}$ witness the path ordering; recall that each point $q_\nu$ in $Q$ corresponds to a finite set $a_{\xi(\nu)}$ or a single vertex $\{y_{\xi(\nu)}\}$ from the $\lambda$-configuration and we identify $q_\nu$ with this set. Moreover, $q_\nu$ must be of the form $y_{\xi(\nu)}$ for every limit $\nu<\lambda$ as $Q$ is concentrated on the main class of $\halff \lambda$. 

Our goal is to define disjoint finite paths $R_\nu$ of colour $i$ in $G\uhp V_i$ so that $q_\nu\subs R_\nu$ while the concatenation $(R_\nu:\nu<\lambda)$ gives a path of colour $i$ in $G\uhp V_i$. 

Construct  $(R_\nu:\nu<\lambda)$ by induction on $\nu<\lambda$ so that
\begin{enumerate}[(i)]
\item $R_\nu$ is a finite path of colour $i$ in $G\uhp V_i$ and $R_\nu\cap (\bigcup \mc X\cup \mc Y)=q_\nu$,
\item $R_\nu\cap R_\mu =\emptyset$ if $\nu<\mu<\lambda$,
\item $R_\nu=q_\nu$ if $q_\nu=\{y_{\xi(\nu)}\}$,
\end{enumerate}
moreover, if $q_\nu=a_{\xi(\nu)}$ then $\nu=\mu+1$ and we make sure that 
\begin{enumerate}
\item[(iv)] the first point of $R_\nu$ is a vertex $v\in a_{\xi(\nu)}$ so that $c(v,y_{\xi(\mu)})=i$, and
\item[(v)] the last point of $R_\nu$ is a vertex $w\in a_{\xi(\nu)}$ so that $c(w,y_{\xi(\nu+1)})=i$.
\end{enumerate}

\begin{figure}[H]%
\centering

\psscalebox{0.7 0.7} 
{

\begin{pspicture}(0,-1.945)(9.76827,1.945)
\psdots[linecolor=black, dotsize=0.24](0.11826935,-1.745)
\psdots[linecolor=black, dotsize=0.24](8.118269,-1.745)
\rput[bl](0.91826934,-1.945){\Large{$y_{\xi(\mu)}$}}
\rput[bl](8.718269,-1.945){\Large{$y_{\xi(\nu+1)}$}}
\rput[bl](1.3182694,0.055){\Large{$v$}}
\rput[bl](7.3182693,0.055){\Large{$w$}}
\psellipse[linecolor=black, linewidth=0.04, dimen=outer](4.3182693,0.055)(4.2,1.0)
\psline[linecolor=black, linewidth=0.04, linestyle=dashed, dash=0.17638889cm 0.10583334cm, arrowsize=0.05291666666666667cm 2.0,arrowlength=1.4,arrowinset=0.0]{->}(0.11826935,-1.745)(1.9182694,-0.345)
\psline[linecolor=black, linewidth=0.04, linestyle=dashed, dash=0.17638889cm 0.10583334cm, arrowsize=0.05291666666666667cm 2.0,arrowlength=1.4,arrowinset=0.0]{->}(6.3182693,0.055)(7.7182693,-1.345)
\psbezier[linecolor=black, linewidth=0.04, linestyle=dashed, dash=0.17638889cm 0.10583334cm, arrowsize=0.05291666666666667cm 2.0,arrowlength=1.4,arrowinset=0.0]{->}(2.3182693,0.055)(2.3182693,0.855)(2.7182693,1.655)(3.5182693,1.455)(4.3182693,1.255)(3.9182694,-0.945)(4.5182695,-1.345)(5.1182694,-1.745)(5.918269,-1.345)(6.1182694,-0.345)
\psdots[linecolor=black, dotsize=0.24](2.3182693,0.055)
\psdots[linecolor=black, dotsize=0.24](6.3182693,0.055)

\rput[bl](5.7182693,1.455){\Large{$a_{\xi(\nu)}$}}
\rput[bl](2.1182694,1.655){\Large{$R_\nu$}}
\end{pspicture}
}

\caption{Constructing $R_\nu$.}
\end{figure}
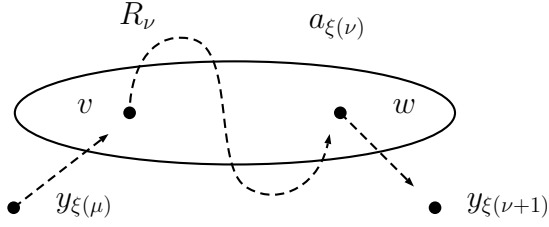

If we can achieve this,  $(R_\nu:\nu<\lambda)$ gives a path of colour $i$ concentrated on $A$.

Note that the only difficulty in this construction is to satisfy the last two requirements; indeed, we always have $\lambda$ many disjoint finite paths of colour $i$ connecting two arbitrary points of any $a_\xi$ (avoiding all other points in question).

How to find the first and last point of $R_\nu$ if $q_\nu=a_{\xi(\nu)}$? As $\nu=\mu+1$ for some $\mu<\lambda$ and by the definition of a path and the colouring $d$ on $\halff \lambda$ we have
$$c(v,y_{\xi(\mu)})=i\text{ for some }v\in a_{\xi(\nu)}$$ and we pick a single such $v\in a_{\xi(\nu)}$ which in turn satisfies $(iv)$ above. 

Second, $d(q_\nu,q_{\nu+1})=i$ hence $\{x\in a_{\xi(\nu)}:c(x,y_{\xi(\nu+1)})=i\}$ has at least two elements so we can pick $$w\in \{x\in a_{\xi(\nu)}:c(x,y_{\xi(\nu+1)})=i\}\setm \{v\}$$ which will satisfy $(v)$ above.

\end{proof}

\end{proof}

\begin{lclaim} \label{clmB}  Let $\kappa$ be an infinite cardinal and $\lambda\leq \cf(\kappa)$. Suppose that $c$ is an $r$-edge colouring of a graph $G=(V,E)$ of type $\halff\kappa$ with $\halff \kappa$-decomposition $(A,B)$ and let $I\subs r$. If for every $\tilde A\in [A]^{<\lambda}$ there is $a\in [A\setm \tilde A]^{<\omega}$ so that $$|B\setm \bigcup\{N(x,i):x\in a, i\in I\}|<\kappa$$ then there is a $\lambda$-configuration $\mc X,\mc Y$ in colours $I$ so that $\bigcup \mc X\subseteq A$.
\end{lclaim}

\begin{proof} We build the sequences $\mc X=\{a_\xi:\xi<\lambda\}$ and $\mc Y=\{y_\xi:\xi<\lambda\}$ inductively so that $$|B\setm \bigcup\{N(x,i):x\in a_\xi, i\in I\}|<\kappa$$ for all $\xi<\lambda$. Given $\{a_\xi:\xi<\zeta\}$ and  $\{y_\xi:\xi<\zeta\}$ we set $\tilde A=\bigcup\{a_\xi:\xi<\zeta\}$. Our assumption gives a finite set $a_\zeta\in [A\setm \tilde A]^{<\omega}$ so that $$|B\setm \bigcup\{N(x,i):x\in a_\zeta, i\in I\}|<\kappa.$$ 

As $X_\zeta=\bigcup\{a_\xi:\xi\leq \zeta\}$ has size $< \lambda\leq \cf(\kappa)$, 
$X_\zeta$ is contained in an initial segment of the $\halff \kappa$ ordering. In turn, $$|N[X_\zeta]|=\kappa.$$ Finally, as $|X_\zeta|<\kappa$, the set $$Y_\zeta=\{y\in N[X_\zeta]:y\in \bigcup\{N(x,i):x\in a_\xi,i\in I\} \text{ for all }\xi\leq \zeta \}$$ has size $\kappa$. Picking $y_\zeta\in Y_\zeta\setm\{y_\xi:\xi<\zeta\}$ finishes the proof.
\end{proof}

\begin{lclaim} \label{clmC} Suppose that $c$ is an $r$-edge colouring of a graph $G=(V,E)$ of type $\halff\kappa$ with $\halff \kappa$-decomposition $(A,B)$. Let $I\subseteq r$ and $X\subseteq A$. If $$|B\setm \bigcup\{N(x,i):x\in a, i\in I\}|=\kappa$$ for all $a\in [X]^{<\omega}$ then there is a partition $\{X_j:j\in r\setm I\}$ of $X$ so that $$|N(x,j)\cap N(x',j)\cap B|=\kappa$$ for all $x,x'\in X_j$ and $j\in r\setm I$.

In particular,  the sets $X_j$ are $\sat \kappa$ in colour $j$ in $X_j\cup B$ and
if $B\subseteq X$ then there is $j\in I\setm r$ so that $X_j$ is $\kappa$-connected in colour $j$.
\end{lclaim}

\begin{proof}  Take a uniform ultrafilter $U$ on $B$ so that $$B\setm \bigcup\{N(x,i):x\in a,i\in I\}\in U$$ for all $a\in [X]^{<\omega}$. Define $X_j=\{x\in X: N(x,j)\in U\}$ for $j<r$ and note that $X_j=\emptyset$ if $j\in I$ while $\{X_j:j\in r\setm I\}$ partitions $X$.

It is clear that $$|N(x,j)\cap N(x',j)\cap B|=\kappa$$ for all $x,x'\in X_j$ and $j\in r\setm I$ and hence $X_j$ is $\sat \kappa$ in colour $j$. Furthermore, if $B\subseteq X$ then there is a $j\in r\setm I$ so that $X_j\cap B\in U$ and hence $X_j$ is $\kappa$-connected in $j$ as $$|N(x,j)\cap N(x',j)\cap X_j|=\kappa$$  for all $x,x'\in X_j$.
\end{proof}

\begin{lclaim}\label{clmD} Suppose that $H$ is of type $\halff \kappa$ with classes $A,B$, $\lambda<\kappa$ and $c$ is an $r$-edge colouring of $H$ with $I\subseteq r$. If there is no $\lambda$-configuration $\mc X, \mc Y$ in colours $I$ with $\bigcup \mc X\subseteq A$ then there is $\tilde A\in [A]^{<\kappa}$ so that $$|B \setm \bigcup\{N(x,i):x\in a, i\in I\}|=\kappa$$ for all $a\in [A\setm \tilde A]^{<\omega}$.
\end{lclaim}

\begin{proof} First, suppose that $\kappa=\cf(\kappa)$. Apply Claim \ref{clmB} to the graph $H$ and $\lambda=\kappa$ and find $\tilde A\in [A]^{<\kappa}$ so that $$|B\setm \bigcup\{N(x,i):x\in a, i\in I\}|=\kappa$$ for all $a\in [A\setm \tilde A]^{<\omega}$. 

Second, suppose that $\kappa>\cf(\kappa)$ and fix an increasing cofinal sequence of regular cardinal $(\kappa_\alpha)_{\alpha<\cf(\kappa)}$ in $\kappa$ so that $\kappa_0>\lambda$. Let $H_\alpha$ denote $H\uhp \kappa_\alpha$; $H_\alpha$ is a graph of type $\halff {\kappa_\alpha}$ and let $A_\alpha,B_\alpha$ denote the two classes. Note that $H_\alpha$ still has no $\lambda$-configuration in colours $I$ and hence we can apply Claim \ref{clmB} to the graph $H_\alpha$ with $\lambda<\kappa_\alpha$: there is $\tilde A_{\alpha}\in [A_\alpha]^{<\lambda}$ so that  $$|B_\alpha\setm \bigcup\{N_{H_\alpha}(x,i):x\in a, i\in I\}|=\kappa_\alpha$$ for all $a\in [A_\alpha\setm \tilde A_{\alpha}]^{<\omega}$. 

Let $\tilde A=\bigcup\{\tilde A_{\alpha}:\alpha<cf(\kappa)\}$ and note that $|\tilde A|\leq cf(\kappa)\cdot \lambda<\kappa$. Now, if $a\in [A\setm \tilde A]^{<\omega}$ then $a\subseteq A_\alpha\setm \tilde A_{\alpha}$ for any large enough $\alpha<cf(\kappa)$ and hence $$|B_\alpha\setm \bigcup\{N_{H_\alpha}(x,i):x\in a, i\in I\}|=\kappa_\alpha$$ for any large enough $\alpha<cf(\kappa)$.
In turn $$|B\setm \bigcup\{N_H(x,i):x\in a, i\in I\}|=\kappa.$$

\end{proof}

\begin{proof}[Proof of Lemma \ref{ml2}] Suppose that condition $(a)$  fails. In particular, for all $i\in I$ there is $\lambda_i<\kappa$ and $X^*_i\subs A$ of size less than $\kappa$ so that there is no path of colour $i$  concentrated on $X$ and order type $\lambda_i$ disjoint from $X^*_i$. Let $ \lambda =\max \{\lambda_i:i\in I\}$ and $X^*=\bigcup\{X^*_i:i\in I\}$. Now, there is no path of colour $i\in I$ and of order type $\lambda$ in $X\setm X^*$ concentrated on $X$. 

Now find a graph $H$ of type $\halff \kappa$ in $G$ with main class $X\setm X^*$ and second class $B'$; this can be done by Observation \ref{Hbuild}. As $(\ih)_{\lambda,|I|}$ holds, Claim \ref{clmA} implies that there is no $\lambda$-configuration $\mc X,\mc Y$ in colours $I$ with $\bigcup \mc X\subseteq X\setm X^*$. 

Apply Claim \ref{clmD} in $H$ and find $\tilde A\in [X\setm X^*]^{<\kappa}$ so that $$|B' \setm \bigcup\{N(x,i):x\in a, i\in I\}|=\kappa$$ for all $a\in [X\setm (X^*\cup \tilde A)]^{<\omega}$. Hence Claim \ref{clmC} applied to $X\setm (X^*\cup \tilde A)$ provides the desired partition and hence clause (b) of Lemma \ref{ml2}.

\end{proof}

\subsection{The existence of monochromatic paths}

We arrived at our first main result which shows, together with Lemma \ref{3line}, the existence of large monochromatic paths in edge coloured graphs of type $\halff \kappa$:

\begin{theorem}\label{halff}
  $(\ih)_\kappa$ holds for all infinite $\kappa$. In particular, if $G$ is a graph of type $\halff \kappa$ with a finite-edge colouring then we can find a monochromatic path of size $\kappa$ concentrated on the main class of $G$. 
\end{theorem}
\begin{proof} We prove $(\ih)_{\kappa,r}$ by induction on $\kappa$ and $r\in \omega$. $(\ih)_{\omega}$ holds  by Lemma \ref{uftrick} and Lemma \ref{3line} so we suppose that $\kappa>\omega$. Also, $(\ih)_{\kappa,1}$ holds by Observation \ref{baseobs}.



Now fix an $r$-edge colouring of a graph $G$ of type $\halff \kappa$ with $\halff \kappa$-decomposition $(A,B)$. 

First, we can suppose that any $X\in[A]^\kappa$ satisfies condition (3) of Lemma \ref{3line} \emph{in all colours}. Indeed, given $X$ we can find a graph $H_X$ of type $\halff \kappa$ in $G$ with main class $X$ (by applying Observation \ref{Hbuild}). Given any colour $i<r$, Lemma \ref{ml1} applied to $H_X$ and colour $i$ tells us that if $X$ fails condition (3) of Lemma \ref{3line} in colour $i$ then we can find a graph $H'_X$ of type $\halff \kappa$ (with main class $X$ minus a set of size $<\kappa$) which is only coloured by $r\setm\{i\}$. Hence we can apply the inductive hypothesis $(\ih)_{\kappa,r-1}$ to $H'_X$ which finishes the proof.

Now, find a \emph{maximal} $I\subseteq r$ so that there is $X\in [A]^{\kappa}$ such that $X$ is $\sat \kappa$ in all colours $i\in I$. Fix such an $I$ and $X$. The following claim finishes the proof.
\begin{lclaim} There is $i\in I$ such that $\utak \kappa i$ holds for $X$.
\end{lclaim}
\begin{proof}Suppose that $X$ fails $\utak \kappa i$ for all $i\in I$. If $|I|<r$ then apply Lemma \ref{ml2} in $G$ to the set $X$ and set of colours $I$. As  $X$ fails $\utak \kappa i$ for all $i\in I$, condition (b) of Lemma \ref{ml2} must hold; in turn, there is a colour $j\in r\setm I$ and a set $X_j\in [X]^\kappa$ so that $X_j$ is $\sat \kappa$ in colour $j$ as well. The fact that $X_j$ is $\sat \kappa$ in each colour $i\in I\cup \{j\}$ contradicts the maximality of $I$.

Hence $I=r$ must hold. Now, for each $i<r$ there is $\lambda_i<\kappa$ and $A^*_i\subs A$ of size less than $\kappa$ so that there is no path of colour $i$ concentrated on $X$ which has order type $\lambda_i$ and is disjoint from $A^*_i$. Let $ \lambda^* =\max \{\lambda_i:i<r\}$ and $A^*=\bigcup\{A^*_i:i<r\}$. Now, there is no path of colour $i<r$ and of order type $\lambda^*$ which is concentrated on $X$ and is disjoint from $A^*$. There is a graph $H$ of type $\halff \kappa$ in $G$ with main class $X\setm A^*$ (by Observation \ref{Hbuild}) and the initial segment $H\uhp \lambda^*$ is of type $\halff {\lambda^*}$. As $(\ih)_{\lambda^*,r}$ holds, we can find a path of type $\lambda^*$ in $H\uhp \lambda^*$ which is concentrated on the main class and hence on $X$. This path is also disjoint from $A^*$ which contradicts our previous assumption.
\end{proof}

\end{proof}


\section{The first decomposition theorem} \label{firstdec}

Our goal now is to prove a path decomposition result for a large class of bipartite graphs which contains $\halff \kappa$.

\begin{definition} Suppose that $\gr$ is a graph, $A\subseteq V$ and $\kappa$ is a cardinal. We say that $A$ is $\cent {\mc A} \kappa$ (in $G$) iff $\mc A=\{(A^i_\alpha)_{\alpha<\lambda_i}:i\in I\}$ for some finite set $I$ so that 
\begin{enumerate}
 \item $A^i_\alpha \subseteq A^i_\beta$ if $\alpha<\beta<\lambda_i$ and $i\in I$,
\item $A\subseteq \bigcup\{A^i_\alpha:\alpha<\lambda_i\}$ for each $i\in I$, and
\item $$|N_G\bigl[\bigcap_{i\in I}A^i_{\alpha_i}\bigr]|\geq\kappa$$ for all $(\alpha_i)_{i\in I}\in \Pi_{i\in I}\lambda_i$.
\end{enumerate}
\end{definition}

In this section, $\mc A$ will always denote a finite set of $\subseteq$-increasing families (indexed by $I$) and $\vec \lambda=(\lambda_i)_{i\in I}$ denotes the length of these families.

Given $\mc A$ and  $\vec{\alpha}=(\alpha_i)_{i\in I}\in \Pi \vec \lambda$ we will write $[\vec{\alpha}]_{\mc A}$ for $\bigcap_{i\in I}A^i_{\alpha_i}$. We call sets of the form $[\vec{\alpha}]_{\mc A}$ an $\mc A$-box.  Furthermore, $\vec \alpha \leq \vec \beta$ will stand for $\alpha_i\leq \beta_i$ for all $i\in I$.

Note that if $A$ is $\cent {\emptyset} \kappa$ then $|N_G[A]|=\kappa$. Also, the main class of a graph $G$ of type $\halff \kappa$ is clearly  $\cent {\mc A} \kappa$ where $\mc A$ is a single increasing cover formed by the initial segments of the $\halff \kappa$ ordering.

Our final goal in this section is to prove the following:

\begin{theorem}\label{centeredcover} Suppose that $\gr$ is a bipartite graph on classes $A,B$ where $|A|=\kappa$. Suppose that $A$ is $\cent {\mc A} \kappa$ for some $\mc A$. Then for any finite edge colouring of $G$, $A$ is covered by disjoint monochromatic paths of different colours.
\end{theorem}

We start with basic observations:

\begin{obs}\label{firstcentobs} Suppose that $A$ is $\cent {\mc A} \kappa$ in a graph $G$ and $\vec \alpha, \vec \beta\in \Pi \vec \lambda$.
\begin{enumerate}
	\item If $\vec \alpha\leq \vec \beta$ then $[\vec \alpha]_{\mc A}\subseteq [\vec \beta]_{\mc A}$ and hence $N_G[[\vec \beta]_{\mc A}]\subseteq N_G[[\vec \alpha]_{\mc A}]$;
	\item $N_G[[\vec \gamma]_{\mc A}]\subseteq N_G[[\vec \alpha]_{\mc A}]\cap N_G[[\vec \beta]_{\mc A}]$ for $\gamma=\max_\leq\{\vec \alpha,\vec \beta\}$;
	\item for every finite $F\subseteq A$ there is an $\mc A$-box $Z$ covering $F$.
\end{enumerate}
In particular, any two points of $A$ are joined by $\kappa$-many disjoint paths of length 2 and hence $A$ is $\sat \kappa$.
\end{obs}

Given a set of increasing covers $\mc A=\{(A^i_\alpha)_{\alpha<\lambda_i}:i\in I\}$ of $A$ and $X\subseteq A$ we write $\mc A\uhp X$ for  $\{(A^i_\alpha\cap X)_{\alpha<\lambda_i}:i\in I\}$.

\begin{obs} Suppose that $A$ is $\cent {\mc A} \kappa$ in a graph $G$. Let $X\subseteq A$, $\vec \alpha\in \vec \lambda$ and $H$ denote the subgraph in $G$ spanned by $X\cup N_G[[\vec \alpha]_{\mc A}]$. Then $X$ is $\cent {\mc A\uhp X} \kappa$ in $H$.

\end{obs}

\begin{obs}\label{cofinobs} Suppose that $(\tilde A^i_\alpha)_{\alpha<\tilde\lambda_i}$ is a cofinal subsequence of $(A^i_\alpha)_{\alpha<\lambda_i}$ for each $i\in I$. Let $\mc A$ and $\tilde {\mc A}$ denote $\{(A^i_\alpha)_{\alpha<\lambda_i}:i\in I\}$ and $\{(\tilde A^i_\alpha)_{\alpha<\tilde\lambda_i}:i\in I\}$ respectively. Then a set of vertices $A$ in a graph $\gr$ is $\cent {\mc A} \kappa$ iff $\cent {\tilde{\mc A}} \kappa$. 

In particular, we can always suppose that $\lambda_i=\cf(\lambda_i)$, each cover is strictly increasing and hence $\lambda_i\leq |A|$.
\end{obs}

We say that a set of vertices $Y\subseteq V$ is $\den {\mc A} \kappa$ iff $$|Y\cap N_G[[\vec \alpha]_{\mc A}]|\geq \kappa$$ for all $\vec\alpha\in \Pi \vec\lambda$.

\begin{obs}\label{denseobs}Suppose that $A$ is $\cent {\mc A} \kappa$ in a graph $G$ and $Y\subseteq V$ is $\den {\mc A} \kappa$. Then
\begin{enumerate}
	\item $Y\cap N_G[[\vec \alpha]_{\mc A}]$ is $\den {\mc A} \kappa$ for all $\vec\alpha\in \Pi \vec\lambda$, and
	\item for any $X\subseteq A$, $X$ is $\cent {\mc A\uhp X} \kappa$ in $G\uhp (X\cup Y)$. 
\end{enumerate}
\end{obs}

Our first non-trivial result connects the previously developed theory of $\halff \kappa$ to this new notion of $\cent {\mc A} \kappa$ subsets.

\begin{lemma}\label{embed} Suppose that $\gr$ is a bipartite graph on classes $A,B$ where $|A|=\kappa$, and $A$ is $\cent {\mc A} \kappa$ for some $\mc A$. Then there is a copy $H$ of the graph $\halff \kappa$ with main class $X\subseteq A$. 
\end{lemma}
\begin{proof} We can suppose that  $\lambda_i=\cf(\lambda_i)\leq \kappa$ for all $i\in I$ by Observation \ref{cofinobs}. Find a maximal $J\subseteq I$ such that there is $\alpha_j<\lambda_j$ for $j\in J$ so that $X_{-1}=\bigcap_{j\in J}A^j_{\alpha_j}$ has size $\kappa$. Note that $J$ might be empty in which case $X_{-1}=A$. Note that $X_{-1}=\bigcup\{X_{-1}\cap A^i_\alpha:\alpha<\lambda_i\}$ is a union of sets of size $<\kappa$ and hence $\cf(\kappa)\leq \lambda_i=\cf(\lambda_i)$ for all $i\in I\setm J$. Without loss of generality, $I\neq J$ otherwise $K_{\kappa,\kappa}$ embeds into $G$. Let us fix $J$, $\alpha_j$ and $A^j_{\alpha_j}$ for $j\in J$ as above.

First, suppose that $\kappa$ is a limit cardinal and take a strictly increasing cofinal sequence $(\kappa_\xi)_{\xi<\cf(\kappa)}$ in $\kappa$. Now inductively find $(\alpha_i(\xi))_{i\in I\setm J}\in \Pi_{i\in I\setm J}\lambda_i$ for $\xi<\cf(\kappa)$ so that $(\alpha_i(\xi))_{i\in I\setm J}\leq (\alpha_i(\zeta))_{i\in I\setm J}$ and $$X_\xi=X_{-1}\cap \bigcap_{i\in I\setm J}A^i_{\alpha_i(\xi)} \text{ has size at least } \kappa_\xi$$ for all $\xi\leq \zeta<\cf(\kappa)$. 

Suppose $(\alpha_i(\xi))_{i\in I\setm J}$ is constructed for $\xi<\zeta$. List $I\setm J$ as $\{i_0, ..., i_m\}$. First, find $\alpha_{i_0}(\zeta)\in \lambda_{i_0}\setm \sup\{\alpha_{i_0}(\xi):\xi<\zeta\}$ such that $$|X_{-1}\cap A^{i_0}_{\alpha_{i_0}(\zeta)}|\geq \kappa_\zeta^{+m}.$$

If we have $\alpha_{i_0}(\zeta),..., \alpha_{i_{k-1}}(\zeta)$ for some $k<m$ so that $$|X_{-1}\cap \bigcap_{l< k} A^{i_l}_{\alpha_{i_l}(\zeta)}|\geq \kappa_\zeta^{+m-k}$$ then find $\alpha_{i_k}(\zeta)\in \lambda_{i_k}\setm \sup\{\alpha_{i_k}(\xi):\xi<\zeta\}$ so that $$|X_{-1}\cap \bigcap_{l\leq k} A^{i_l}_{\alpha_{i_l}(\zeta)}|\geq \kappa_\zeta^{+m-k-1}.$$ This finishes the inductive construction.

Let $X=\bigcup\{X_\xi:\xi<\cf(\kappa)\}$ and note that $X_\xi$ has size $<\kappa$ and $|N[X_\xi]|=\kappa$ since $X_\xi$ is an $\mc A$-box for each $\xi<\cf(\kappa)$. Observation \ref{Hbuild} can be applied now to find a copy $H$ of $\halff \kappa$ with main class $X$.

If $\kappa=\mu^+$ we inductively find $(\alpha_i(\xi))_{i\in I\setm J}\in \Pi_{i\in I\setm J}\lambda_i$ for $\xi<\cf(\kappa)$ so that $$X_\xi=X_{-1}\cap \bigcap_{i\in I\setm J}A^i_{\alpha_i(\xi)} \text{ has size } \mu$$ and $X_\xi \subsetneq X_\zeta$ for all $\xi\leq \zeta<\kappa$. First, note that $\lambda_i=\kappa$ for all $i\in I\setm J$. As before, suppose $(\alpha_i(\xi))_{i\in I\setm J}$ is constructed for $\xi<\zeta$ and list $I\setm J$ as $\{i_0, ..., i_m\}$. Fix $x\in X_{-1}\setm \bigcup\{X_\xi:\xi<\zeta\}$. Suppose we have  $\alpha_{i_0}(\zeta),..., \alpha_{i_{k-1}}(\zeta)$ for some $k<m$ so that $$|X_{-1}\cap \bigcap_{l< k} A^{i_l}_{\alpha_{i_l}(\zeta)}|=\mu$$ and $x\in X_{-1}\cap \bigcap_{l< k} A^{i_l}_{\alpha_{i_l}(\zeta)}$. We claim that there is $\alpha_{i_k}(\zeta)\in \kappa\setm \sup\{\alpha_{i_k}(\xi):\xi<\zeta\}$ so that  $$|X_{-1}\cap \bigcap_{l\leq k} A^{i_l}_{\alpha_{i_l}(\zeta)}|=\mu$$ and $x\in X_{-1}\cap \bigcap_{l\leq k} A^{i_l}_{\alpha_{i_l}(\zeta)}$. Indeed, we cannot write a set of size $\mu$ as an increasing union of ${\mu^+}$ sets of size $<\mu$. 

Finally, let $X=\bigcup\{X_\xi:\xi<\kappa\}$. As before, $X_\xi$ has size $<\kappa$ and $|N[X_\xi]|=\kappa$ for each $\xi<\kappa$. Hence Observation \ref{Hbuild} can be applied to find a copy $H$ of $\halff \kappa$ with main class $X$.
\end{proof}

The next lemma shows that the property of being ``$\cent {\mc A} \kappa$ for some $\mc A$'' is inherited by subgraphs in a strong sense.

\begin{lemma}\label{cover} Suppose that $\gr$ is a bipartite graph on classes $A,B$ and $A$ is $\cent {\mc A} \kappa$ for some $\mc A$. Suppose that $H$ is a subgraph of $G$ such that $$|N_G[x]\setm N_H[x]|<\kappa$$ for all $x\in X=V(H)\cap A$. Then there is a finite $\mc A'\supseteq \mc A\uhp X$ so that $X$ is $\cent {\mc {A'}} \kappa$ in $H$.
\end{lemma}
\begin{proof} We define $\mc A'$ by extending $\mc A\uhp X$ with at most two new covers depending on the size of $X$ and on $\kappa$ being a limit or successor cardinal.

 First, if $X$ happens to have size $\kappa$ then let $(X^0_\alpha)_{\alpha<\cf(\kappa)}$ be an increasing sequence of subsets of $X$ of size less than $\kappa$ with union $X$. We put $(X^0_\alpha)_{\alpha<\cf(\kappa)}$ into $\mc A'$ if $|X|=\kappa$.

Second, if $\kappa$ is a limit cardinal then let us take a strictly increasing cofinal sequence $(\kappa_\alpha)_{\alpha<\cf(\kappa)}$ in $\kappa$ and let $$X^1_\alpha=\{x\in X:|N_G[x]\setm N_H[x]|\leq\kappa_\alpha \}$$ for $\alpha<\cf(\kappa)$. We put $(X^1_\alpha)_{\alpha<\cf(\kappa)}$ into $\mc A'$ as well if $\kappa$ is a limit.

Let us show that $\mc A'$ works. If $Z\subseteq X$ is an $\mc A'$-box then $|Z|<\kappa$ and there is $\lambda<\kappa$ such that $|N_G[x]\setm N_H[x]|\leq \lambda$ for all $x\in Z$. In particular $$|\bigcup_{x\in Z} N_G[x]\setm N_H[x]|\leq |Z|\cdot \lambda <\kappa.$$ Also, $|N_G[Z]|=\kappa$ as $Z$ is contained in an $\mc A$-box. Hence the set $$N_H[Z]=N_G[Z]\setm \bigl(\bigcup_{x\in Z} N_G[x]\setm N_H[x]\bigr)$$ has size $\kappa$.

\end{proof}

Lemma \ref{cover} is the reason we work with this new class of bipartite graphs instead of $\halff \kappa$. Note that if $X$ is a subset of the main class of $\halff \kappa$ then $X$ is not necessarily covered by a subgraph isomorphic to $\halff \lambda$ for some $\lambda\leq \kappa$.

The next lemma is our final preparation to the proof of Theorem \ref{centeredcover}.

\begin{lemma}\label{indedges} Suppose that $\gr$ is a bipartite graph on classes $A,B$ where $|A|=\kappa$ and $A$ is $\cent {\mc A} \kappa$ for some $\mc A$. Let $c$ be a finite edge colouring of $G$ and suppose that $G_0$ is a subgraph of $G$ with classes $V(G_0) \cap A=A_0$ and $V(G_0)\cap B=B_0$. If 
\begin{enumerate}
	\item $|A_0|=\kappa$ and $B_0$ is  $\den {\mc A} \kappa$ in $G$, and
	\item $|\ran(c\uhp E(G_0))|$ is minimal among subgraphs $G_0$ of $G$ satisfying (1)
	\end{enumerate}
	then
	\begin{enumerate}
	\setcounter{enumi}{2}
	\item for every $i\in \ran(c\uhp E(G_0))$ and every $X\in [A_0]^\kappa$ there is a set of $\kappa$ independent edges $\{\{x_\alpha,y_\alpha\}:\alpha<\kappa\}\subseteq c^{-1}(i)$ so that $\{x_\alpha:\alpha<\kappa\}\subseteq X$ and $\{y_\alpha:\alpha<\kappa\}$ is $\den {\mc A} \kappa$ in $G$.
\end{enumerate} 
\end{lemma}

\begin{proof} Suppose $\mc A=\{(A^i_\alpha)_{\alpha<\lambda_i}:i\in I\}$ and $\vec \lambda = (\lambda_i)_{i\in I}$ as before. Again, we can suppose that $\Pi \vec \lambda$ has size $\leq \kappa$ by Observation \ref{cofinobs}. Take a subgraph $G_0$ of $G$ which satisfies (1) and suppose that (3) fails; we will show that $|\ran(c\uhp E(G_0))|$ is not minimal i.e. (2) fails. 

Let $i\in \ran(c\uhp E(G_0))$ and $X\in [A_0]^\kappa$ witness that condition (3) fails.  Enumerate $\Pi \vec \lambda$ as $\{\vec \alpha(\xi):\xi<\kappa\}$ such that each $\vec \alpha \in  \Pi \vec \lambda$ appears $\kappa$ times. Start inductively building independent edges  $\{\{x_\xi,y_\xi\}:\xi<\zeta\}\subseteq c^{-1}(i)$ from $X$ so that $y_\xi\in B_0\cap N_G[[\vec \alpha(\xi)]_{\mc A}]$. There must be a $\zeta<\kappa$ such that we cannot pick $\{x_\zeta,y_\zeta\}$. That is, every edge from $X\setm \{x_\xi:\xi<\zeta\}$ to $B_0\cap N_G[[\vec \alpha(\zeta)]_{\mc A}]\setm \{y_\xi:\xi<\zeta\}$ is not coloured $i$. Let $A_1= X\setm \{x_\xi:\xi<\zeta\}$ and $B_1=B_0\cap N_G[[\vec \alpha(\zeta)]_{\mc A}]\setm \{y_\xi:\xi<\zeta\}$. It is easy to see that $G_1=G_0\uhp A_1\cup B_1$ satisfies (1); indeed, $A_1$ has size $\kappa$ and Observation \ref{denseobs} implies that $B_1$ is $\den {\mc A} \kappa$ in $G$. Finally, $i\notin \ran(c\uhp E(G_1))$ implies  $|\ran(c\uhp E(G_1))|<|\ran(c\uhp E(G_0))|$ and we are done.

\end{proof}

\begin{proof}[Proof of Theorem \ref{centeredcover}]
We prove the statement by induction on $r\geq 1$ for every $\gr$, $\mc A$ and $c$ simultaneously. 

First, suppose $r=1$. Lemma \ref{embed} implies that we can find a copy $H$ of $\halff \kappa$ in $G$ with main class $X\subseteq A$. Hence, by Theorem \ref{halff}, there is a path $P$ of size $\kappa$ which is concentrated on $X$. As $A$ is $\sat \kappa$ (by Observation \ref{firstcentobs}) we can cover $A$ by a single path in $G$ using Lemma \ref{lego}.

Now, suppose we proved the statement for $r-1$ and fix $\gr$, $\mc A$ and an $r$-edge colouring $c$. We will show that there is a colour $i<r$ and a path $P$ of colour $i$ in $G$ such that $A\setm P$ is one class of a bipartite subgraph $G_1$ of $G$ so that 
\begin{enumerate}[(i)]
	\item $V(G_1)\cap P=\emptyset$,
	\item $i\notin \ran(c\uhp E(G_1))$,
	\item $A\setm P$ is $\cent {\mc {A'}} \kappa$ in $G_1$ for some finite $\mc A' \supseteq \mc A$.
\end{enumerate} 
Once we find such a path $P$ and subgraph $G_1$, applying the inductive hypothesis finishes the proof.

First, take a subgraph $G_0$ of $G$ with classes $V(G_0) \cap A=A_0$ and $V(G_0)\cap B=B_0$ such that
\begin{enumerate}
	\item $|A_0|=\kappa$ and $B_0$ is  $\den {\mc A} \kappa$ in $G$, and
	\item $|\ran(c\uhp E(G_0))|$ is minimal among subgraphs of $G$ satisfying (1).
	\end{enumerate}
	
	Find a partition of $B_0$ into $B^0_0$ and $B^1_0$ so that both sets are $\den {\mc A} \kappa$ in $G$. Let $G^l_0=G_0\uhp (A_0\cup B^l_0)$ for $l<2$ and note that $\ran(c\uhp E(G^1_0))=\ran(c\uhp E(G_0))$ by (2). Hence, by Lemma \ref{indedges}, for every $i\in \ran(c\uhp E(G_0))$ and every $X\in [A_0]^\kappa$ there is a set of $\kappa$ independent edges $\{\{x_\alpha,y_\alpha\}:\alpha<\kappa\}\subseteq c^{-1}(i)$ so that $\{x_\alpha:\alpha<\kappa\}\subseteq X$ and $\{y_\alpha:\alpha<\kappa\}\subseteq B^1_0$ is $\den {\mc A} \kappa$ in $G$.

 Now, embed a copy $H$ of $\halff \kappa$ in $G^0_0$ using Lemma \ref{embed}. By Theorem \ref{halff}, we can find  $i<r$ and a set $X$ in the main class of $H$ which satisfies all three conditions of Lemma \ref{3line} in colour $i$. By (2), there is a set of $\kappa$ independent edges $\{\{x_\alpha,y_\alpha\}:\alpha<\kappa\}\subseteq c^{-1}(i)$ in $G^1_0$ so that $\{x_\alpha:\alpha<\kappa\}\subseteq X$ and $Y=\{y_\alpha:\alpha<\kappa\}\subseteq B^1_0$ is $\den {\mc A} \kappa$ in $G$.

\begin{figure}[H]%
\centering

\psscalebox{0.8 0.8} 
{
\begin{pspicture}(1,-2.485)(12.99,2.485)
\definecolor{colour0}{rgb}{0.8,0.8,0.8}
\psellipse[linecolor=black, linewidth=0.04, fillstyle=solid,fillcolor=colour0, dimen=outer](8.68,0.995)(1.7,0.3)
\psline[linecolor=black, linewidth=0.04](12.58,-0.465)(1.0,-0.465)(1.0,-2.465)(12.58,-2.465)
\psframe[linecolor=black, linewidth=0.04, dimen=outer](5.4,1.535)(2.4,0.535)
\psframe[linecolor=black, linewidth=0.04, dimen=outer](11.8,1.535)(6.2,0.535)
\psframe[linecolor=black, linewidth=0.04, dimen=outer](9.2,-0.865)(2.2,-2.065)
\rput[bl](0.0,-1.465){\Large{$A$}}
\rput[bl](9.84,-1.605){\Large{$A_0$}}
\rput[bl](1.36,0.955){\Large{$B^0_0$}}
\rput[bl](12.4,0.935){\Large{$B^1_0$}}
\psellipse[linecolor=black, linewidth=0.04, fillstyle=solid,fillcolor=colour0, dimen=outer](5.1,-1.465)(2.5,0.4)
\psline[linecolor=black, linewidth=0.04, linestyle=dashed, dash=0.17638889cm 0.10583334cm](3.2,-1.465)(3.0,0.935)(3.8,-1.465)(3.6,0.935)(4.6,-1.465)(4.2,0.935)(6.0,-1.465)(5.0,0.935)
\psdots[linecolor=black, dotsize=0.16](5.2,-1.465)
\psdots[linecolor=black, dotsize=0.16](5.8,-1.465)
\psdots[linecolor=black, dotsize=0.16](6.4,-1.465)
\psdots[linecolor=black, dotsize=0.16](7.6,0.935)
\psdots[linecolor=black, dotsize=0.16](8.2,0.935)
\psdots[linecolor=black, dotsize=0.16](9.0,0.935)
\psline[linecolor=black, linewidth=0.04](5.2,-1.465)(7.6,0.935)
\psline[linecolor=black, linewidth=0.04](5.8,-1.465)(8.2,0.935)
\psline[linecolor=black, linewidth=0.04](6.4,-1.465)(9.0,0.935)
\rput[bl](8.0,-1.665){\Large{$X$}}
\rput[bl](10.74,0.855){\Large{$Y$}}
\psbezier[linecolor=black, linewidth=0.04, linestyle=dotted, dotsep=0.10583334cm, arrowsize=0.05291666666666667cm 2.0,arrowlength=1.4,arrowinset=0.0]{->}(8.24,2.315)(7.0764923,2.545615)(6.315542,2.393765)(5.96,2.0086586)(5.604458,1.6235522)(5.651494,0.7601294)(6.26,-0.025)
\rput[bl](8.54,1.995){\large{edges of colour $i$}}
\psbezier[linecolor=black, linewidth=0.04, linestyle=dotted, dotsep=0.10583334cm, arrowsize=0.05291666666666667cm 2.0,arrowlength=1.4,arrowinset=0.0]{->}(1.98,2.315)(1.096492,2.265615)(0.655542,1.9337649)(0.54,1.5086585)(0.42445797,1.0835522)(0.9514942,0.12012937)(2.6,0.055)
\rput[bl](2.3,2.075){\large{copy of $H_{\kappa,\kappa}$}}
\end{pspicture}
}
\caption{Preparing the cover of $A$.}
\label{coverfig}
\end{figure}
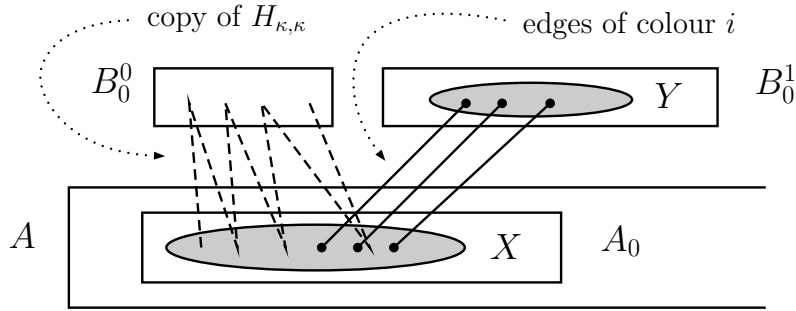

 Let $$\bar X=X\cup \{x\in A:|N_G(x,i)\cap Y|=\kappa\}.$$ Note that $\bar X$ is still $\sat \kappa$ in colour $i$ in $G$.

\begin{claim} There is $Y_1\in[Y]^\kappa$ so that $Y_1$ is $\den {\mc A}\kappa$ in $G$ and $$|N_G(x,i)\cap Y\setm Y_1|=\kappa$$ for all $x\in \bar X\setm X$.
\end{claim}
\begin{proof} The proof goes by an easy induction of length $\kappa$.
\end{proof}

Note that $\bar X$ still satisfies all three condition of Lemma \ref{3line} in $V\setm Y_1$ and $|N_G(x,i)\cap Y_1|<\kappa$ for all $x\in A\setm \bar X$. Now find a path $P$ of colour $i$ in $V\setm Y_1$ which covers $\bar X$; this can be done by Lemma \ref{3line}. Note that $A\setm P$ is $\cent {\mc A\uhp A\setm P} \kappa$ in $G\uhp (A\setm P\cup Y_1)$ and the subgraph $$G_1=(A\setm P\cup Y_1, c^{-1}(r\setm\{i\}))$$

satisfies the assumptions of Lemma \ref{cover}. In particular, $A\setm P$ is $\cent {\mc {A'}} \kappa$ for some finite $\mc A'\supseteq \mc A\uhp A\setm P$ in $G_1$. This finishes the proof.

\end{proof}

\section{The main decomposition theorem} \label{seconddec}

At this point, it would be rather easy to show (using Theorem \ref{centeredcover}) that every $\kappa$-complete graph is covered by $2r$ (not necessarily disjoint) monochromatic paths. However, we prove the following much stronger theorem which is the main result of this paper:

\begin{theorem}\label{maindecomp} Suppose that $c$ is a finite-edge colouring of a $\kappa$-complete graph $G=(V,E)$. Then the vertices can be partitioned into disjoint monochromatic paths of different colours.
\end{theorem}

\begin{proof} We can suppose $\kappa>\omega$. First, note that any $\kappa$-complete graph $\gr$ is actually $|V|$-complete; thus it suffices to prove the theorem for $\kappa$-complete graphs of size $\kappa$. The next arguments will be reminiscent of the proof of Theorem \ref{halff}.

\begin{lclaim} \label{2of3} Suppose that $c$ is an $r$-edge colouring of $G$ with $r\in \omega$. Then  there is $A\in[V]^\kappa$ and $i<r$ so that $A$ is $\kappa$-connected in colour $i<r$ and satisfies $\utak \kappa i$ in $G\uhp A$ at the same time.
\end{lclaim}
\begin{proof}
Suppose there is no such $A$. By a finite induction, we construct sets $A_0\supseteq A_1 \supseteq \dots$ of size $\kappa$ and a 1-1 sequence $i_0,i_1,\dots$ in $r$ so that $A_k$ is $\kappa$-connected in colour $i_k$. 

Suppose $k=0$. As $G$ is of type $\halff \kappa$ with main class $V$ (see Observation \ref{comptype}) we can apply Claim \ref{clmC} with $I=\emptyset$ and $X=V$. We find a colour $i_0$ and a set $A_0$ of size $\kappa$ which is connected in colour $i_0$.

Suppose $k<r-1$ and we defined $A_{k}$. As $A_j$ must fail $\utak \kappa {i_j}$ in $G\uhp A_j$ for all $j\leq k$, we have $A^*_j\in [A_j]^{<\kappa}$ and $\lambda_j<\kappa$ such that there is no path $P$ in colour $i_j$ in $G\uhp (A_j\setm A^*_j)$ which is concentrated on $A_j$ and has order type $\lambda_j$. Let $A^*=\bigcup\{A_{j}^*:j\leq k\}$ and  $\lambda=\max\{\lambda_{j}:{j}\leq k\}$. Note that $H=G\uhp (A_k\setm A^*)$ is of type $\halff \kappa$ with main class $A_k\setm A^*$ and there is no $\lambda$-configuration in colours $I=\{i_{j}:j\leq k\}$ inside $H$. Indeed, otherwise Claim \ref{clmA} would imply that there is a path of type $\lambda$ in colour $i_{j}$ inside in $G\uhp (A_j\setm A^*_j)$ for some $j\leq k$ (recall that $A_k\setm A^*$ is $\sat \kappa$ in colour $i_j$ in $G\uhp (A_j\setm A^*_j$). Hence, Claim \ref{clmD} and \ref{clmC} implies that we can find a set $A_{k+1}\in [A_k]^\kappa$ and colour $i_{k+1}\in r\setm \{i_{j}:j\leq k\}$ so that $A_{k+1}$ is $\kappa$-connected in colour $i_{k+1}$.

Suppose we defined $A_{r-1}$. By assumption, $A_{r-1}$ fails $\utak \kappa i$ in $G\uhp A_{r-1}$  for all $i<r$. However, Theorem \ref{halff} implies the existence of a monochromatic path of size $\kappa$ in some colour $i<r$ which in turn implies that $\utak \kappa {i}$ must hold for some $i<r$ by Observation \ref{utakeq}.

\end{proof}


\begin{lclaim} \label{3of3} There are sets $A,Y\in [V]^\kappa$ and $i<r$ so that $Y\subseteq A$ and $A\setm Z$ satisfies all three conditions of Lemma \ref{3line} in colour $i$ in $G\uhp (A\setm Z)$ for all  $Z\subseteq Y$. Moreover, we can suppose that $A$ is a \emph{maximal} $\kappa$-connected subset.
\end{lclaim}
In particular, $A\setm Z$  is a single path of colour $i$ for \emph{every choice of $Z\subseteq Y$} by Lemma \ref{3line}. 

\begin{proof} This claim is proved by induction on $r$. If $r=1$ then let $A=V$ and let $Y\subseteq A$ such that $A\setm Y$ and $Y$ has size $\kappa$. Given $Z\subs Y$, we know that $G\uhp (A\setm Z)$ is $\kappa$-complete and hence of type $\halff \kappa$ with main class $A\setm Z$. Hence, by Observation \ref{baseobs}, $A\setm Z$  satisfies all three conditions of Lemma \ref{3line} in $G\uhp (A\setm Z)$.

Suppose that $r>1$. Now, we can suppose that any set $X\in [V]^\kappa$ satisfies condition (3) of Lemma \ref{3line} in $G\uhp X$ in \emph{all colours} $i<r$. Indeed, note that $G\uhp X$ is of type $\halff \kappa$ with main class $X$ and suppose $X$ fails condition (3) of Lemma \ref{3line} in $G\uhp X$ in some colour $i<r$. Now Lemma \ref{ml1} implies that there is $\tilde X\in [X]^{<\kappa}$ so that $X\setm \tilde X$ is covered by a subgraph $H$ of $G\uhp X$ of type $\halff \kappa$ with main class $X\setm \tilde X$ so that $i\notin \ran(c\uhp E(H))$. Without loss of generality $V(H)\cap \tilde X=\emptyset$ i.e. $V(H)$ is the main class of $H$. Hence Observation \ref{comptype} implies that we can find a $\kappa$-complete subgraph $G'$ in $H$; the inductive hypothesis can be applied to $G'$ as $i\notin \ran(c\uhp E(G'))$.

Now, take $A\in[V]^\kappa$  which is a maximal $\kappa$-connected subset in some colour $i<r$ and satisfies $\utak \kappa i$ in $G\uhp A$; this can be done by Claim \ref{2of3}. It is easy to see that we can find $Y\in [A]^\kappa$ so that $A\setm Z$ is still $\kappa$-connected in colour $i$ and satisfies $\utak \kappa i$ in $G\uhp (A\setm Z)$ for any $Z\subseteq Y$. Indeed, we construct $Y$ by an induction of length $\cf(\kappa)$: let $\{\kappa_\alpha:\alpha<\cf(\kappa)\}$ be a cofinal sequence of cardinals in $\kappa$ ($\kappa_\alpha=\lambda$ if $\kappa=\lambda^+$) and let $A_\alpha\in [A]^{\kappa_\alpha}$ increasing so that $A=\bigcup\{A_\alpha:\alpha<\cf(\kappa)\}$. Define sets $Y_\alpha\in [A]^{\kappa_\alpha}$, $W_\alpha\in [A]^{\kappa_\alpha}$ for $\alpha<\cf(\kappa)$ so that $Y_\alpha\cap W_\beta=\emptyset$ for all $\alpha,\beta< \cf(\kappa)$ and
\begin{enumerate}
	\item there are $\kappa_\alpha$ many disjoint paths of order type $\kappa_\alpha$ and colour $i$ in $G\uhp W_\alpha$ concentrated on $A$, and
	\item for any $u\neq v \in A_\alpha$, there are $\kappa_\alpha$ many disjoint paths of colour $i$ from $v$ to $u$ in $W_\alpha\cup\{u,v\}$.
\end{enumerate}
It is clear that $Y=\bigcup\{Y_\alpha:\alpha<\cf(\kappa)\}$ is as desired. As $A\setm Y$ satisfies (3) of Lemma \ref{3line}, we are done.

\end{proof}

Find $A,Y\subs V$ and $i<r$ as in Claim \ref{3of3} with $A$ being a maximal $\kappa$-connected subset in colour $i$. Let $X=V\setm A$. Let $H$ denote the bipartite subgraph of $G$ on classes $X,Y$ where $\{v,w\}\in E(H)$ iff $v\in Y,w\in X$ and $c(v,w)\neq i$. Note that $$|Y\setm N_H(x)|<\kappa \text{ for all }x\in X;$$ otherwise $ A\cup \{x\}$ is still $\kappa$-connected in colour $i$. 

If $K$ denotes the complete bipartite graph on classes $X,Y$ then $X$ is $\cent \emptyset \kappa$ in $K$. Furthermore, the subgraph $H$ of $K$ satisfies the conditions of Lemma \ref{cover} and hence there is a  finite $\mc A'$ so that $X$ is $\cent {\mc A'} \kappa$ in $H$. 

By Theorem \ref{centeredcover}, there is a set of disjoint monochromatic paths $\mc Q$ in $H$ which covers $X$; recall that $i\notin \ran(c\uhp E(H))$ and hence none of the paths in $\mc Q$ has colour $i$.  Note that $Z=(\cup \mc Q)\setm X\subseteq Y$ and hence $V\setm \cup \mc Q= A\setm Z$ satisfies all three conditions of Lemma \ref{3line} in colour $i$ in $G\uhp (V\setm \cup \mc Q)$. In particular, $V\setm \cup \mc Q$ is a single path $P$ in colour $i$ and hence $\mc Q\cup \{P\}$ is a decomposition of $V(G)$ into disjoint monochromatic paths of different colours.

\end{proof}

\section{Open problems}\label{probl_sec}

It is a natural question if one can extend our result to infinite complete bipartite graphs:

\begin{conj}\label{biconj} Suppose that the edges of an infinite complete bipartite graph are coloured with $r\in \omega$ colours. Then we can partition the vertices into $2r-1$ disjoint monochromatic paths.
\end{conj}

Note that Theorem \ref{centeredcover} implies that we can find a cover (not necessarily disjoint) by $2r$ monochromatic paths. Conjecture \ref{biconj} appeared for finite graphs in \cite{Po2} and is proved for the countably infinite case in \cite{thesis}. \\


One can consider the monochromatic path decomposition problem when the edges of the complete graph are coloured with infinitely many colours. There is a simple limitation of proving a monochromatic path decomposition theorem, namely one might not be able to decompose the vertices into sets so that each set is connected in some colour. This problem was investigated by A. Hajnal, P. Komj\'ath, L. Soukup and I. Szalkai in \cite{conn_decomp}.

Let us remain in the realm of $\omega$-colourings for now. A possible first step towards a general result could be looking at the following Ramsey-theoretic problem: let $\mb P$ denote the class of cardinals $\kappa$ such that for every edge colouring $c:[\kappa]^2\to \omega$ of $K_\kappa$ there is a monochromatic path of size $\kappa$.  It is easy to colour the edges of $K_\omg$ with $\omega$ colours without monochromatic cycles and hence $\omega_1\notin \mb P$. Furthermore, note that if $\kappa$ satisfies the partition relation $\kappa\to (\kappa)^2_\omega$ then $\kappa\in \mb P$ hence many large cardinals are in $\mb P$.

\begin{prob} Can we prove that $\mb P$ is non empty in ZFC? If so, what is $\min \mb P$? 
\end{prob}

 $\omega_2$ or $\mathfrak{c}^+$ seem to be natural candidates for $\min \mb P$. 

\bibliographystyle{plain}
\bibliography{thesis}

\begin{thebibliography}{10}

\bibitem{monopathI}
M.~Elekes, D.~T. Soukup, L.~Soukup, and Z.~Szentmiklóssy.
\newblock Decompositions of edge-colored infinite complete graphs into
  monochromatic paths.
\newblock {\em submitted to Discr. Math., arXiv:1502.04955}, 2015.

\bibitem{EGyP}
P.~Erdős, A.~Gy{\'a}rf{\'a}s, and L.~Pyber.
\newblock Vertex coverings by monochromatic cycles and trees.
\newblock {\em Journal of Combinatorial Theory, Series B}, 51(1):90--95, 1991.

\bibitem{Gy}
A.~Gy{\'a}rf{\'a}s.
\newblock Covering complete graphs by monochromatic paths.
\newblock In {\em Irregularities of partitions}, pages 89--91. Springer, 1989.

\bibitem{bestbound}
A.~Gy{\'a}rf{\'a}s, M.~Ruszink{\'o}, G.~N. S{\'a}rk{\"o}zy, and
  E.~Szemer{\'e}di.
\newblock An improved bound for the monochromatic cycle partition number.
\newblock {\em Journal of Combinatorial Theory, Series B}, 96(6):855--873,
  2006.

\bibitem{GyS1}
A.~Gy{\'a}rf{\'a}s and G.~N. S{\'a}rk{\"o}zy.
\newblock Monochromatic path and cycle partitions in hypergraphs.
\newblock {\em The Electronic Journal of Combinatorics}, 20(1):P18, 2013.

\bibitem{conn_decomp}
A.~Hajnal, P.~Komj{\'a}th, L.~Soukup, and I.~Szalkai.
\newblock Decompositions of edge colored infinite complete graphs.
\newblock In {\em Colloq. Math. Soc. J{\'a}nos Bolyai}, volume~52, pages
  277--280, 1987.

\bibitem{kunen}
K.~Kunen.
\newblock Set theory. an introduction to independence proofs.
\newblock {\em Studies in Logic and the Foundations of Mathematics, 102.
  North-Holland Publishing Co., Amsterdam-New York.}, 1980.

\bibitem{Po2}
A.~Pokrovskiy.
\newblock Partitioning edge-coloured complete graphs into monochromatic cycles
  and paths.
\newblock {\em Journal of Combinatorial Theory, Series B}, 106:70--97, 2014.

\bibitem{R}
R.~Rado.
\newblock Monochromatic paths in graphs.
\newblock {\em Ann. Discrete Math.}, 3:191--194, 1978.

\bibitem{thesis}
D.~T. Soukup.
\newblock Colouring problems of {E}rdős and {R}ado on infinite graphs.
\newblock {\em PhD thesis, University of Toronto, Department of Mathematics},
  2015.

\bibitem{soukel}
L.~Soukup.
\newblock Elementary submodels in infinite combinatorics.
\newblock {\em Discrete Math.}, 311(15):1585--1598, 2011.

\end{thebibliography}

\end{document}